\documentclass[12pt]{amsart}%
\usepackage[centertags]{amsmath}
\usepackage{amsmath}
\usepackage{amsthm}
\usepackage{graphicx}%
\usepackage{amsfonts}%
\usepackage{amssymb}
\usepackage{color}
\usepackage[all]{xy}
\usepackage[colorlinks=true,hyperindex=true,bookmarks=true]{hyperref}
\hypersetup{
    colorlinks,
    citecolor=blue,
    filecolor=purple,
    linkcolor=red,
    urlcolor=blue
} 
\newtheorem{theorem}{Theorem}[section]
\newtheorem{corollary}[theorem]{Corollary}
\newtheorem{lemma}[theorem]{Lemma}
\newtheorem{proposition}[theorem]{Proposition}
\theoremstyle{definition}
\newtheorem{definition}[theorem]{Definition}
\theoremstyle{remark}
\newtheorem{remark}[theorem]{Remark}

\numberwithin{equation}{section}
\newcommand{\ds}{\displaystyle}
\newcommand{\R}{\mathbb R}

\newcommand{\N}{\mathbb N}
\newcommand{\Z}{\mathbb Z}
\newcommand{\norm}[1]{\left\Vert#1\right\Vert}
\newcommand{\abs}[1]{\left\vert#1\right\vert}
\newcommand{\set}[2]{\left\{ #1 \,;\, #2 \right\}}
\newcommand{\eps}{\varepsilon}

\begin{document}
\title{Nemytskii operators between Stepanov almost periodic or almost automorphic function spaces}
\author{Philippe CIEUTAT{}$^1$}
%
%
\setcounter{footnote}{-1}
\renewcommand{\thefootnote}{\alph{footnote}}
\footnote{ {}$^1$ Laboratoire de Math\'ematiques de Versailles, UVSQ, CNRS, Universit\'e Paris-Saclay, 78035 Versailles, France.  E-mail address: philippe.cieutat@uvsq.fr}

\begin{abstract}
We study the superposition operators (also called Nemytskii operators) between spaces of almost periodic (respectively almost automorphic) functions in the sense of Stepanov. We state new results on the superposition, notably we give a necessary and sufficient condition for that these operators are well-defined and continuous.

\end{abstract}

%
\maketitle

\noindent
{\bf 2010 Mathematic Subject Classification:} Superposition operators, Stepanov almost periodic functions, Stepanov almost automorphic functions, Bochner transform.
\vskip2mm
\noindent
{\bf Keywords:} 42A75, 43A60.

\vskip10mm
%

\section{Introduction}
\label{i}

In this  work, we study some properties of superposition operators, also called Nemytskii operators, on the space of  Stepanov almost periodic or Stepanov almost automorphic functions. 
Denote by $\mathcal{F}(\R,X)$ (resp. $\mathcal{F}(\R,Y)$) a space of functions from $\R$  into a Banach $X$ (resp. $Y$).
For a given function $f:\R\times X\to Y$, the  Nemytskii operator associated to $f$ is the map $\mathcal{N}_f:\mathcal{F}(\R,X)\to \mathcal{F}(\R,Y)$  defined by the formula
$\mathcal{N}_f(u)(t) = f(t,u(t))$ for $u\in\mathcal{F}(\R,X)$ and $t\in\R$.
The Nemytskii operators play an important role in the theory of differential and integral equations. 
First studies of this kind of operators are presumably due to Nemytskii (see the preface of \cite{V}); 
that is why such operators are sometime called Nemytskii operators.
\vskip2mm
When $\mathcal{F}(\R,X)$ (resp. $\mathcal{F}(\R,Y)$) designs the space of almost periodic or almost automorphic functions in the sense of Stepanov with values in $X$ (resp. $Y$),
our aim is to answer these questions: what assumptions should check  $f$ for that
\vskip 2 mm
Q1- the Nemytskii operator $\mathcal{N}_f$ maps $\mathcal{F}(\R,X)$ into $\mathcal{F}(\R,Y)$, that is to mean
$[t\mapsto f(t,u(t))]\in\mathcal{F}(\R,Y)$, for all $u\in\mathcal{F}(\R,X)$ (composition result), 
\vskip 2 mm
Q2- the Nemytskii operator $\mathcal{N}_f$ is continuous?
\vskip 2 mm
Many authors have partially answered to the question Q1 in
\cite{Di-Hu-Li-Xi, Di-Li-Xi, Lo-Di, NG-Di}.
For that they use a Lipschitzian condition on $f$ and a compactness condition on the function $u$.  These results are of the type:
when $f$ satisfies a Lipschitzian condition, if  $u\in\mathcal{F}(\R,X)$ and the range of $u$ is relatively compact, then $[t\mapsto f(t,u(t))]\in\mathcal{F}(\R,Y)$.
Then in a recent article \cite[Theorem 2.11]{Be-Ch-Me-Ra-Sm}, Bedouhene et al. have deleted this last compactness condition.  
Without the Lipschitz condition,  Andres et al. has answered to the question Q1 in the particular case where $f$ does not depend of $t$ in \cite[Lemma 3.2]{An-Pe} and in the general case in \cite[Proposition 3.4]{An-Pe}. We give an improvement of these two results (Corollary \ref{cor10} and Theorem \ref{th11}).
\vskip 2 mm
In the almost periodicity case  in the sense of Bohr, that is $\mathcal{F}(\R,X) = AP(\R,X)$ and $\mathcal{F}(\R,Y) = AP(\R,Y)$ are spaces of almost periodic functions in the sense of Bohr, the condition $f(\cdot,x)$  is almost periodic in the sense of Bohr, for all $x\in X$, is not  sufficient   to obtain the  assertion Q1
(cf. \cite[Chapter 2, p. 16]{Fi}). 
Yoshizawa has given
a definition of almost periodicity on the function $f$, the so-called  almost periodicity in $t$ uniformly for $x\in X$. With this definition, when $X$ and $Y$ are of finite dimension, Yoshizawa has answered to the question Q1 in \cite[Definition 2.1, p. 5, Theorem 2.7, p. 16]{Yo}. 
Then this last result is generalized for general Banach spaces by Blot et al. in \cite[Theorem 3.5]{Bl-Ci-Ng-Pe}, and it is  also established that the Nemytskii operator is continuous. In \cite[Theorem 3.5 \& 3.12]{Bl-Ci-Ng-Pe}, it is stated a necessary and sufficient conditions for that  the Nemytskii operator $\mathcal{N}_f$ maps $AP(\R,X)$ into $AP(\R,Y)$ and it is continuous.
Among these necessary and sufficient conditions, there is firstly the function $f$ is almost periodic in $t$ uniformly for $x\in X$,  secondly the restriction of the Nemytskii to $X$:
$\mathcal{N}_f:X\to AP(\R,Y)$ with $\mathcal{N}_f(x)=f(\cdot,x)$, 
 is well defined and it is continuous (Theorem \ref{th9}). 
\vskip 2 mm
The goal of this work is to give a necessary and sufficient condition for that the Nemytskii operator $\mathcal{N}_f$ maps $\mathcal{F}(\R,X)$ into $\mathcal{F}(\R,Y)$ and it is continuous where $\mathcal{F}(\R,X)$ and $\mathcal{F}(\R,Y)$ are spaces of almost periodic functions in the sense of Stepanov (Theorem \ref{th3}). 
This necessary and sufficient condition is:
$f(\cdot,x)$ is almost periodic in the sense of Stepanov for all $x\in X$ and
the restriction of the Nemytskii operator to the space of $1$-periodic functions in $L_{loc}^p(\R,X)$, with values in the space of bounded functions in the sense of Stepanov  is well-defined and continuous.
The almost automorphic case
is also treated (Theorem \ref{th2}).
\vskip 2 mm
Our work is organized as follows: in Section \ref{iii} we give some
notations and  definitions about almost periodic functions and almost automorphic functions, then  we recall known results on the Nemytskii operators which will be used. In Section \ref{ii} we built  a left inverse of the Bochner transform  which permits to state that the range of Bochner transform is closed and admits a topological complement. This left inverse  will be used to state the main result of the following section.  In Section \ref{iv} we give a necessary and sufficient condition to obtain the continuity of Nemytskii operators between almost periodic and almost automorphic spaces in the sense of Stepanov, which  permits to generalize some known results of
\cite{Be-Ch-Me-Ra-Sm, Di-Li-Xi, Lo-Di}.
In Section \ref{v} we extend some results of Danilov in \cite{Da, Da1} from Stepanov almost periodic to Stepanov almost automorphic  functions. 
These results will be  used in the following section. 
In Section \ref{vi} we state two equivalent results (Theorem \ref{th5} and \ref{th11}) which improve and generalize all the known results on the composition of Stepanov almost periodic or almost automorphic  functions.
We  give sufficient conditions to obtain the continuity of Nemytskii operators between Stepanov spaces. The assumptions are directly on the function $f$, unlike of Section \ref{iv} where assumptions are on  Nemytskii operators built on $f$.
By giving an example, in Section \ref{vii}, we explain why Theorem \ref{th5} and \ref{th11} provide an improvement and a generalization of  results in \cite{An-Pe, Be-Ch-Me-Ra-Sm, Di-Hu-Li-Xi, Di-Li-Xi, Lo-Di, NG-Di}.

\section{Notation and definitions}
\label{iii}

\subsection{Notation}


In this section we give the notations and definitions that will be used and we recall some known results on the Nemytskii operators.

\vskip 2 mm

$\R$, $\Z$ and $\N$ stand for the real numbers, the integers and  the natural integers  respectively. 
\vskip 2 mm
When $t\in\R$, we denote by {\it $[t]$ the integer part and $\{t\}$  the fractional part of $t$}, then $t=[t]+\{t\}$ with $[t]\in\Z$ and $0\leq \{t\}<1$.
\vskip 2 mm
When $A$ is a Lebesgue measurable set of $\R$, we denote by $\rm{meas}\left( A\right)$ the Lebesgue measure of $A$.
\vskip 2 mm
Let $X$ be  a Banach space.
\vskip 2 mm
When $T$ is a metric space, $C(T,X)$ denotes the space of all continuous mappings from $T$ into $X$. If $T$ is compact, then $C(T,X)$ endowed with the supremum norm $\ds\norm{u}_{\infty} = \sup_{t\in T}\norm{u(t)}$ is a Banach space.
\vskip 2 mm
Let $BC(\R,X)$ be  the space of all bounded and continuous maps from $\R$ into $X$. Endowed with the supremum norm $\ds \norm{u}_{\infty} = \sup_{t\in\R}\norm{u(t)}$, $BC(\R,X)$ is a Banach space. 
\vskip 2 mm
Let $1\leq p<+\infty$. We denote by $L^p(a,b;X)$ the space of all functions from $(a,b)$ into $X$ $p$-integrable in the sense of Bochner 
 with respect to the Lebesgue measure on  the bounded interval $(a,b)$,
with the convention that any two functions equal almost everywhere (a.e.) specify the same element of $L^p(a,b;X)$.
Endowed with the usual norm $\ds \norm{\omega}_{L^p(a,b;X)} = \left(\int_a^b\norm{\omega(\theta}^p\,d\theta\right)^{\frac{1}{p}}$, $L^p(a,b;X)$ is a Banach space. 
We denote by $\norm{\cdot}_{L^p}$ the usual norm of $L^p(0,1;X)$.
For the Bochner integral we reefer to \cite{AB, Di-Uh}.
$L_{loc}^p(\R,X)$ stands for 
 the space of all functions $u:\R\to X$ such that the restriction of $u$ to every bounded interval $(a,b)$ is in $L^p(a,b;X)$.
 \vskip 2 mm
We define $L^{\infty}(\R,X)$ to be the space of $X$-valued essentially bounded functions.
\vskip 2 mm
If $\mathcal{F}(E,F)$ designs a set of maps from $E$ into $F$, as usual we denote by $\mathcal{F}(E)$ the set $\mathcal{F}(E,F)$ when $F=\R$, for example $L^p(0,1) = L^p(0,1;\R)$.

\subsection{Almost periodic and almost automorphic functions}

A set $D\subset\R$ is said to be {\it relatively dense} in $\R$ if: 
$\ds\exists \ell>0$, $\forall\alpha>0$, such that $D\cap [\alpha,\alpha+\ell]\not=\emptyset$.
\vskip 1 mm
A continuous function $u:\R\to X$ is said to be {\it almost periodic (in the sense of Bohr)} if  for each $\eps > 0$, the set of $\eps$-almost periods of $u$:
$$\mathcal{P}(u,\eps)=\set{\tau\in\R}{\sup_{t \in \mathbb{R}} \norm{u (t + \tau) - u (t)}_X\leq  \eps}$$
 is relatively dense in $\R$.  We denote the space of all such functions by $AP(\R,X)$. It is a Banach subspace of $BC(\R,X)$. 
 For some preliminary results on almost periodic functions, we refer to the book of Corduneanu \cite{Co}.
\vskip 2 mm
A continuous function $u:\R\to X$ is said to be {\it almost automorphic}  if for all sequence of real numbers $(t_{k}^{\prime})_{k\in\N}$ admits  a subsequence  denoted by $(t_{k})_{k\in\N}$ such that
\[
\forall t\in\R,\quad \underset{k\rightarrow\infty}{\lim}u(t+t_{k})  = v(t)\quad\text{and}\quad
\underset{k\rightarrow\infty}{\lim}v(t-t_{k})=u(t) .
\]
Then we have $v\in L^{\infty}(\R,X)$.
We denote the space of all such functions by $AA(\R,X)$. It is a Banach subspace of $BC(\R,X)$. 
We have the following inclusions which are strict
 $$AP(\R,X)\subset AA(\R,X) \subset BC(\R,X) .$$
For some preliminary results on almost automorphic functions, we refer to the book of N'Gu\'er\'ekata \cite{NG}. 
\vskip 2 mm
Let $1\leq p<+\infty$. For $u\in L_{loc}^p(\R,X)$, 
we denote by $u^b$ the {\it Bochner transform} of $u$ defined by
$u^b(t)(\theta)= u(t+\theta)$, for $t\in\R$ and $\theta\in(0,1)$. $u^b(t)$ is regarded as a function with values in the space $L^p(0,1;X)$.
$BS^p(\R,X)$ denotes the space  of {\it bounded functions in the sense of Stepanov of exponent $p$} which is defined by
$$BS^p(\R,X)=\set{u\in L_{loc}^p(\R,X)}{u^b\in L^{\infty}(\R,L^p(0,1;X))} .$$
Note that for every $u\in L_{loc}^p(\R,X)$, 
the function $u^b$ is continuous  (by construction), 
then the space $BS^p(\R,X)$ may be also written 
$$BS^p(\R,X)=\set{u\in L_{loc}^p(\R,X)}{u^b\in BC(\R,L^p(0,1;X))} .$$
The space $BS^p(\R,X)$ endowed by the norm 
\begin{equation}
\label{eq9}
\norm{u}_{S^p}=\sup_{t\in\R}\left(\int_0^1\norm{u(t+\theta}^p\,d\theta\right)^{\frac{1}{p}} = \sup_{t\in\R}\norm{u^b(t)}_{L^p}\quad\text{ for }u\in BS^p(\R,X) 
\end{equation}
is a Banach space. 
\vskip 2 mm
$S_{ap}^p(\R,X)$ denotes the space  of {\it almost periodic functions in the sense of Stepanov of exponent $p$} which is defined by
$$S_{ap}^p(\R,X)=\set{u\in L_{loc}^p(\R,X)}{u^b\in AP(\R,L^p(0,1;X))} .$$
$S_{ap}^p(\R,X)$ is a Banach subspace of $BS^p(\R,X)$.
We have the following strict inclusion $AP(\R,X) \subset S_{ap}^p(\R,X)$.
For some preliminary results on bounded or almost periodic functions in the sense of Stepanov, 
we refer to the book of Amerio and Prouse \cite{Am-Pr} and that of Pankov \cite{Pa}.
We also quote the paper of Andres et al. \cite{An-Be} which discusses the relationships between various definition of almost periodic functions.
\vskip 2 mm
$S_{aa}^p(\R,X)$ denotes the space  of {\it almost automorphic functions in the sense of Stepanov of exponent $p$} which is defined by
$$S_{aa}^p(\R,X)=\set{u\in L_{loc}^p(\R,X)}{u^b\in AA(\R,L^p(0,1;X))} .$$
$S_{aa}^p(\R,X)$ is a Banach subspace of $BS^p(\R,X)$.
We have the following strict  inclusions $AA(\R,X) \subset S_{aa}^p(\R,X)$ and 
$S_{ap}^p(\R,X)\subset S_{aa}^p(\R,X)\subset BS^p(\R,X)$.
For some preliminary results on almost automorphic functions in the sense of Stepanov, we refer to the paper of Casarino \cite{Ca} or of N'Gu\'er\'ekata-Pankov \cite{NG-Pa}.

\subsection{Almost periodic and almost automorphic sequences}


Denote by $X^\Z$ the set of all two-sided sequences $u=(u_n)_{n\in\Z}$ with values in the Banach space $X$. 
$\ell^{\infty}(\Z,X)$ denotes the set of all sequences $u=(u_n)_{n\in\Z}$  of $X^\Z$ which are  bounded.
Endowed with the supremum norm $\ds\norm{u}_{\infty} = \sup_{n\in\Z}\norm{u_n}$, $\ell^{\infty}(\Z,X)$ is a Banach space. 
\vskip 2 mm
A set $D\subset\Z$ is said to be {\it relatively dense} in $\Z$ if: 
$\exists N\in\N\setminus\{0\}$, 
$\forall m\in\Z$, such that 
$D\cap \{m, \cdots,m+N\}\not=\emptyset$.
\vskip 1 mm
A sequence $u=(u_n)_{n\in\Z}\in X^\Z$ is said to be {\it almost periodic} if  for each $\eps > 0$, the set of $\eps$-almost periods of $u$:
$$\mathcal{P}(u,\eps)=\set{p\in\Z}{\sup_{n \in\Z} \norm{u_{n+p}  - u_n}_X\leq  \eps}$$
 is relatively dense in $\Z$.  We denote the space of all such sequences by $AP(\Z,X)$. It is a Banach subspace of $\ell^{\infty}(\Z,X)$.
 For some preliminary results on almost periodic sequences, we refer to the book of Corduneanu \cite{Co}.
\vskip 2 mm
A sequence $u=(u_n)_{n\in\Z}\in X^{\Z}$  is said to be {\it almost automorphic}  if all sequence of integer numbers $(p_{k}^{\prime})_{k\in\N}$ admits  a subsequence  denoted by $(p_{k})_{k\in\N}$ such that 
\[
\forall n\in\Z,\quad \underset{k\rightarrow\infty}{\lim}u_{n+p_{k}} = v_n \quad\text{and}\quad
\underset{k\rightarrow\infty}{\lim}v_{n-p_{k}}=u_n .
\]
Then we have $v\in \ell^{\infty}(\Z,X)$. We denote the space of all such sequences by $AA(\Z,X)$. It is a Banach subspace of $\ell^{\infty}(\Z,X)$. 
 We have the following strict inclusions:
 $$AP(\Z,X)\subset AA(\Z,X) \subset \ell^{\infty}(\Z,X) .$$
For some preliminary results on almost automorphic sequences, we refer to the book of 
Diagana \cite{Di1}.

\subsection{Nemytskii operators}


Here we recall some known results on Nemytskii operators which will be used in the sequel. Let be $X$ and $Y$ two Banach spaces, $p$ and $q$ be two real numbers in $[1,+\infty)$.

\vskip 2 mm
We become to recall a result on Nemytskii operators between Lebesgue spaces in the context of separable Banach spaces. We say that a function $f:(0,T)\times X\to Y$  (with $T>0$) is a {\it Carath\'eodory function} if:
\vskip 2 mm
a) for all $x\in X$, the map $f(\cdot,x)$ is  measurable from $(0,T)$ into $Y$;
\vskip 2 mm
b) for 
a.e. $t\in (0,T)$,
 the map $f(t,\cdot)$ is continuous from $X$ into $Y$.
\vskip 2 mm
We consider the Nemytskii operator on $f$ defined by 
\begin{equation}
\label{eq53}
N_f:L^p(0,T;X)\to L^q(0,T;Y) \quad\text{with}\quad
N_f(\omega)(t)=f(t,\omega(t))\quad\text{ for }t\in(0,T) .
\end{equation}
In the context of separable Banach spaces, the 
following result is from Lucchetti and Patrone.

\begin{theorem} \cite[Theorem 3.1]{Lu-Pa}
\label{th6}
Let be $X$ and $Y$ two separable Banach spaces and $f:(0,T)\times X\to Y$  be is a Carath\'eodory function.
Then the Nemytskii operator $N_f$ defined by \eqref{eq53} maps $L^p(0,T;X)$ into $L^q(0,T;Y)$ if and only if 
there exist $a>0$ and $b\in L^q(0,T)$ such that for all $x\in X$ and a.e.  $t\in (0,T)$
\begin{equation*}
\norm{f(t,x)}\leq a\norm{x}^{\frac{p}{q}}+b(t) . 
\end{equation*}
In this case the Nemytskii operator $N_f$ is continuous.
\end{theorem}

\vskip 2 mm
From a map $f:\R\times X\to Y$ we consider the Nemytskii operator of $f$ in the almost periodic case
\begin{equation*}
\mathcal{N}_{f}:AP(\R,X)\to AP(\R,Y)
\end{equation*}
defined by 
\begin{equation}
\label{eq1}
\mathcal{N}_{f}(u)(t)=f(t,u(t)) \quad \text{  for } t\in\R .
\end{equation}
\vskip 2 mm
According to Yoshizawa \cite{Yo}, a continuous function $f:\R\times X\to Y$ is said to be {\it almost periodic in $t$ uniformly with respect to $x$} if for for each compact set $K\subset X$ and  for each $\eps > 0$, the set
$$\set{\tau\in\R}{\sup_{t \in \R}\sup_{x \in K} \norm{f (t + \tau,x) - f (t,x)}\leq  \eps}$$
 is relatively dense in $\R$.  We denote the space of all such functions by $AP_U(\R\times X,Y)$.
\vskip 1 mm
A function $f\in AP_U(\R\times X,Y)$ if and only if 
for all $x\in X$, $f(\cdot,x)\in AP(\R,Y)$  and $f$ satisfies
\begin{equation}
\label{eq60}
\left\{
\begin{array}
[c]{l}
\text{ For each compact set K }\subset X, \quad  \forall \eps>0, \, \exists \delta>0,\\
\text{ \ \ \ \ \ \ \ \ \ \ \ \ \ \ \ \ \ \ \ \ \ }\\
\forall x_1 ,  x_2\in K, \, \forall t\in\R, \quad\norm{x_1-x_2}\leq \delta \implies \norm{f(t,x_1)-f(t,x_2)}\leq\eps 
\end{array}
\right.  \end{equation}
\cite[Lemma 2.6]{Ci-Fa-Ng}.

\begin{theorem} \cite{Bl-Ci-Ng-Pe, Ci-Fa-Ng}
\label{th9}
Let $f:\R\times X\to Y$ be a map. The following assertions are equivalent.
\vskip2mm
{\bf i)} The Nemytskii operator $\mathcal{N}_{f}$ defined by \eqref{eq1} maps $AP(\R,X)$ into $AP(\R,Y)$ and it is continuous.
\vskip2mm
{\bf ii)} $f\in AP_U(\R\times X,Y)$.
\vskip2mm
{\bf iii)}
For all $x\in X$, $f(\cdot,x)\in AP(\R,Y)$  and $f$ satisfies \eqref{eq60}.
\vskip2mm
{\bf iv)} The map $\Phi:X\to AP(\R,Y)$ defined by  $\Phi(x)=f(\cdot,x)$ for $x\in X$
is well-defined and continuous. 
\vskip2mm
{\bf v)} For each compact set $K\subset X$, the map $\tilde{f}_K:\R\to C(K,Y)$ defined by  $\tilde{f}_K(t)=f(t,\cdot)$ for $t\in\R$,
is almost periodic: $\tilde{f}_K\in AP(\R,C(K,Y))$. 
\end{theorem}

\begin{remark}
Remark that $X$ may be regarded as a Banach subspace of $AP(\R,X)$, then the map $\Phi$ defined in iv) is the restriction of the Nemytskii operator $\mathcal{N}_{f}$ on the subspace $X$.
\vskip 2 mm
{\bf i) $\Longleftrightarrow$ ii)} is proved in \cite[Theorem 3.5 \& 3.12]{Bl-Ci-Ng-Pe} and
{\bf ii) $\Longleftrightarrow$ iii)} in  \cite[Lemma 2.6]{Ci-Fa-Ng}.
Since $X$ is a metric space, {\bf iv)} $\Longleftrightarrow$
the restriction of $\Phi$ to each compact set $K\subset X$ is well-defined and uniformly continuous, which is equivalent to {\bf iii)}. {\bf ii) $\Longleftrightarrow$ v)} results of \cite[Lemma 3.3]{Bl-Ci-Ng-Pe}.
\end{remark}

\vskip 2 mm
From a map $f:\R\times X\to Y$ we consider the Nemytskii operator of $f$ in the almost automorphic case
\begin{equation*}
\mathcal{N}_{f}:AA(\R,X)\to AA(\R,Y)
\end{equation*}defined by \eqref{eq1}.

\vskip 2 mm
According to Blot et al. \cite[Section 2]{Bl-Ci-Ng-Pe}, a continuous function $f:\R\times X\to Y$ is said to be {\it almost automorphic in $t$ uniformly with respect to $x$} if
for all $x\in X$, $f(\cdot,x)\in AA(\R,Y)$  and $f$ satisfies \eqref{eq60}.
We denote the space of all such functions by $AA_U(\R\times X,Y)$. 
\vskip 1 mm
Let $f:\R\times X\to Y$ be a continuous function.
From \cite[Theorem 3.14]{Ci-Fa-Ng}, we have he following statement: $f\in AA_U(\mathbb{R}\times X,Y)$ if and only   if for
each compact set $K\subset X$ and if for all sequence of real numbers $(t_{k}^{\prime})_{k\in\N}$ admits  a subsequence  denoted by $(t_{k})_{k\in\N}$ and there exists a function $g:\mathbb{R}\times X \to Y$ such that for all $t\in\R$
\begin{equation*}
\lim_{k\rightarrow +\infty}\sup_{x\in K} \parallel f(t+t_k,x) - g(t,x) \parallel =0  \quad\text{and}\quad
\lim_{k\rightarrow +\infty}\sup_{x\in K} \parallel g(t-t_k,x) - f(t,x) \parallel =0 .
\end{equation*}

\begin{theorem} \cite{Bl-Ci-Ng-Pe, Ci-Fa-Ng}
\label{th8}
Let $f:\R\times X\to Y$ be a map. The following assertions are equivalent.
\vskip2mm
{\bf i)} The Nemytskii operator $\mathcal{N}_{f}$ defined by \eqref{eq1} maps $AA(\R,X)$ into $AA(\R,Y)$ and it is continuous.
\vskip2mm
{\bf ii)} $f\in AA_U(\R\times X,Y)$.
\vskip2mm
{\bf iii)}
For all $x\in X$, $f(\cdot,x)\in AA(\R,Y)$  and $f$ satisfies \eqref{eq60}.
\vskip2mm
{\bf iv)} The map $\Phi:X\to AA(\R,Y)$ defined by  $\Phi (x)=f(\cdot,x)$ for $x\in X$
is well-defined and continuous.
\vskip2mm
{\bf v)} For each compact set $K\subset X$, the map $\tilde{f}_K:\R\to C(K,Y)$ defined by  $\tilde{f}_K(t)=f(t,\cdot)$ for $t\in\R$,
is almost automorphic: $\tilde{f}_K\in AA(\R,C(K,Y))$.
\end{theorem}

\begin{remark}
{\bf i) $\Longleftrightarrow$ ii)} is proved in \cite[Theorem 9.6]{Bl-Ci-Ng-Pe}. 
Since $X$ is a metric space, {\bf iv)} $\Longleftrightarrow$
the restriction of $\Phi$ to each compact set $K\subset X$ is well-defined and uniformly continuous
$\Longleftrightarrow$
{\bf iii)}, which is the 
definition of $f\in AA_U(\R\times X,Y)$. Then {\bf iv) $\Longleftrightarrow$ iii) $\Longleftrightarrow$ ii)}. {\bf ii) $\Longleftrightarrow$ v)} results of \cite[Theorem 3.14]{Ci-Fa-Ng}.
\end{remark}

\section{Bochner transform}
\label{ii}

In this Section we build  a left inverse of the Bochner transform which permits to state that the range of Bochner transform is closed and admits a topological complement. This left inverse  will be used to state the main result of Section \ref{iv}. 
To build a left inverse of the Bochner transform we begin to give  a description of  the range under the Bochner transform of almost periodic  or almost automorphic space in the sense of Stepanov. 
 \vskip 2 mm
Let $X$ be a Banach space and $1\leq p<+\infty$. 
The Bochner transform
\begin{equation*}
B:BS^p(\R,X)\to BC(\R,L^p(0,1;X)) \text{ with } Bu=u^b
\end{equation*}
is an isometry which is not surjective. We also have  $\ds B(S_{ap}^p(\R,X))\varsubsetneqq AP(\R,L^p(0,1;X))$ and $\ds B(S_{aa}^p(\R,X))\varsubsetneqq AA(\R,L^p(0,1;X))$.
\vskip 2 mm
Now we  give  a description of  ranges of $S_{ap}^p(\R,X)$ and $S_{aa}^p(\R,X)$
 under the Bochner transform. For that we introduce  the discrete Bochner transform:
\begin{equation}
\label{eq2}
D:BS^p(\R,X)\to \ell^{\infty}(\Z,L^p(0,1;X)),\,Du(n)=u^b(n)\text{ for }n\in\Z.
\end{equation}
Remark that $D$ may be also defined by $D=R\circ B$ where is $R$ the restriction operator defined by
\begin{equation}
\label{eq3}
R:BC(\R,L^p(0,1;X))\to \ell^{\infty}(\Z,L^p(0,1;X)) \text{ with } Ru(n)=u(n)\text{ for }n\in\Z .
\end{equation}

\begin{proposition}
\label{th1}
{\bf i)} The discrete Bochner transform $D$ defined by \eqref{eq2} is an homeomorphism of 
$BS^p(\R,X)$ onto $\ell^{\infty}(\Z,L^p(0,1;X))$ (linear, bijective and bicontinuous) with
\begin{equation}
\label{eq40}
D^{-1}U(t) = U([t])(\{t\}) \quad \text{ for } U\in \ell^{\infty}(\Z,L^p(0,1;X)) \text{ and } t\in\R .
\end{equation}
\vskip2mm
{\bf ii)} The operator $D$  is also an homeomorphism of $S_{aa}^p(\R,X)$ onto $AA(\Z,L^p(0,1;X))$: 
$$D\left(S_{aa}^p(\R,X)\right)=AA(\Z,L^p(0,1;X)) .$$
\vskip2mm
{\bf iii)} The operator $D$  is also an homeomorphism of $S_{ap}^p(\R,X)$ onto $AP(\Z,L^p(0,1;X))$: 
$$D\left(S_{ap}^p(\R,X)\right)=AP(\Z,L^p(0,1;X)) .$$
\end{proposition}

The main difficulty to proof  Proposition \ref{th1} is to state that  $AA(\Z,L^p(0,1;X))$ is included in $D\left(S_{aa}^p(\R,X)\right)$; for that we use the following lemma.

\begin{lemma}
\label{lem4}
Let  $w_1$, $w_2\in BS^p(\R,X)$. For all $p\in\Z$, $s$ and $\tau\in\R$ such that $\abs{\tau}\leq 1$, one has
\begin{equation*}
\norm{(w_1)^b(s+p+\tau)-(w_2)^b(s+\tau)}_{L^p}^p \leq \sum_{j=[s]-1}^{[s]+2}
\norm{(w_1)^b(j+p)-(w_2)^b(j)}_{L^p}^p .
\end{equation*}
\end{lemma}

\begin{proof}
From the definition of the Bochner transform, we have
\begin{equation*}
\norm{(w_1)^b(s+p+\tau)-(w_2)^b(s+\tau)}_{L^p}^p
= \int_{s+\tau}^{s+\tau+1} \norm{w_1(p+\sigma)-w_2(\sigma)}^p\,d\sigma
\end{equation*}
and
\begin{equation*}
\sum_{j=[s]-1}^{[s]+2}
\norm{(w_1)^b(j+p)-(w_2)^b(j)}_{L^p}^p
=  \int_{[s]-1}^{[s]+3} \norm{w_1(p+\sigma)-w_2(\sigma)}^p\,d\sigma .
\end{equation*}
The conclusion results of the following inequality
\begin{equation*}
\int_{s+\tau}^{s+\tau+1} \norm{w_1(p+\sigma)-w_2(\sigma)}^p\,d\sigma
\leq \int_{[s]-1}^{[s]+3} \norm{w_1(p+\sigma)-w_2(\sigma)}^p\,d\sigma .
\end{equation*}
\end{proof}

\vskip    2mm \noindent\textit{Proof of Proposition \ref{th1}.}
{\bf i)} For $u\in L_{loc}^p(\R,X)$, we have 
$$\int_t^{t+1}\norm{u(s)}^p\,ds
\leq  \int_{[t]}^{[t]+2}\norm{u(s)}^p\,ds
\leq2\sup_{n\in\Z}\int_{n}^{n+1}\norm{u(s)}^p\,ds$$
where the supremum may be equal to $+\infty$.
It follows that
\begin{equation}
\label{eq7}
\forall u\in L_{loc}^p(\R,X) , \quad \sup_{t\in\R}\left(\int_0^1\norm{u(t+\theta)}^p\,d\theta\right)^{\frac{1}{p}}
\leq
2^{\frac{1}{p}}\sup_{n\in\Z}\left(\int_0^1\norm{u(n+\theta)}^p\,d\theta\right)^{\frac{1}{p}}  .
\end{equation}
From \eqref{eq7} and with $\norm{Du(n)}_{L^p} = \left(\int_0^1\norm{u(t+\theta)}^p\,d\theta\right)^{\frac{1}{p}}$ for $n\in\Z$, we obtain
\begin{equation}
\label{eq6}
\forall u\in BS^p(\R,X) , \quad \sup_{n\in\Z}\norm{Du(n)}_{L^p} \leq\norm{u}_{S^p} \leq  2^{\frac{1}{p}} \sup_{n\in\Z}\norm{Du(n)}_{L^p} .
\end{equation}
 From \eqref{eq6}, it follows that the linear  operator $D$ is bounded and injective.
\vskip 2 mm
We set $\ds U\in\ell^{\infty}(\Z,L^p(0,1;X))$ and consider the function $u:\R\to X$ defined by 
\begin{equation*}
u(t)=U([t])(\{t\}) .
\end{equation*}
To state  that $D$ is surjective and \eqref{eq40}, we prove that $u\in BS^p(\R,X)$ and $Du=U$.
For $n\in\Z$ and $t\in\R$ such that $n\leq t <n+1$, we have $u(t)=U(n)(t-n)$, then $u\in L^p(n,n+1;X)$ for all $n\in\Z$ which implies  that $u\in L_{loc}^p(\R,X)$. 
From \eqref{eq7}, it follows
$\ds\sup_{t\in\R}\left(\int_0^1\norm{u(t+\theta)}^p\,d\theta\right)^{\frac{1}{p}}
\leq 2^{\frac{1}{p}}\sup_{n\in\Z}\norm{U(n)}_{L^p}<+\infty$, therefore $u\in BS^p(\R,X)$. By using definitions of $D$ and $u$, we have
$Du(n)(\theta)=u(n+\theta)=U(n)(\theta)$ for $n\in\Z$ and $\theta\in(0,1)$,
then $Du=U$.
The bicontinuity of $D$ results of \eqref{eq6}. 
\vskip2mm
{\bf ii)} For $u\in S_{aa}^p(\R,X)$, we have $Bu=u^b\in AA(\R,L^p(0,1;X))$. 
A direct consequence of the definition of an almost automorphic sequence and an almost automorphic function is that the restriction of an almost automorphic function to the integer numbers is an almost automorphic sequence, 
then $Du=(R\circ B)u\in AA(\Z,L^p(0,1;X))$ where is $R$ the restriction operator defined by \eqref{eq3},
consequently $D\left(S_{aa}^p(\R,X)\right)\subset AA(\Z,L^p(0,1;X))$.
\vskip1mm
For the reciprocal inclusion, let $U\in AA(\Z,L^p(0,1;X))$. By using  i), we can assert the existence  and uniqueness of $u\in BS^p(\R,X)$ such that $Du=U$. Now  prove that
$u\in S_{aa}^p(\R,X)$. Let $(t_{k}^{\prime})_{k\in\N}$ be a sequence of real numbers.   
Then $([t_{k}^{\prime}])_{k\in\N}$ is a sequence of integer numbers.
Since $(U(n))_{n\in\Z}$ is an almost automorphic sequence and $(\{t_{k}^{\prime}\})_{k\in\N}$ is a bounded sequence, there exist $V\in \ell^{\infty}(\Z,L^p(0,1;X))$, $\tau_*\in\R$ and a 
subsequence of $(t_{k}^{\prime})_{k\in\N}$, denoted by $(t_{k})_{k\in\N}$ such that 
\begin{equation}
\label{eq29}
\forall n\in\Z , \quad \lim_{k\to+\infty}\norm{U(n+[t_k])-V(n)}_{L^p} =0, 
\end{equation}
\begin{equation}
\label{eq30}
\forall n\in\Z , \quad \lim_{k\to+\infty}\norm{V(n-[t_k])-U(n)}_{L^p} =0, 
\end{equation}
\begin{equation*}
\lim_{k\to+\infty} \{t_k\}=\tau_*. 
\end{equation*}
By using  i), we can assert the existence  and uniqueness of $v\in BS^p(\R,X)$ such that $Dv=V$. 
We fix $t\in\R$. From Lemma \ref{lem4}, with $w_1=u$, $w_2=v$, $s=t$, $p=[t_k]$ and $\tau=\{t_k\}$, we obtain
\begin{equation*}
\norm{u^b(t+[t_k]+\{t_k\})-v^b(t+\{t_k\})}_{L^p}^p \leq \sum_{j=[t]-1}^{[t]+2}
\norm{u^b(j+[t_k])-v^b(j)}_{L^p}^p .
\end{equation*}
Since $t_k=[t_k]+\{t_k\}$, $j$ and  $[t_k]\in\Z$,  $u^b=U$ and $v^b=V$ on $\Z$, we have
\begin{equation*}
\norm{u^b(t+t_k)-v^b(t+\{t_k\})}_{L^p}^p \leq \sum_{j=[t]-1}^{[t]+2}
\norm{U(j+[t_k])-V(j)}_{L^p}^p 
\end{equation*}
and from \eqref{eq29}, we obtain
\begin{equation*}
\lim_{k\to+\infty}\norm{u^b(t+t_k)-v^b(t+\{t_k\})}_{L^p}^p =0 .
\end{equation*}
By using  the continuity of the Bochner transform  of $v$ and $\lim_{k\to+\infty} \{t_k\}=\tau_*$, we also have
\begin{equation*}
\lim_{k\to+\infty}\norm{v^b(t+\{t_k\})-v^b(t+\tau_*)}_{L^p}^p =0 .
\end{equation*}
Then
\begin{equation}
\label{eq34}
\lim_{k\to+\infty}\norm{u^b(t+t_k)-v^b(t+\tau_*)}_{L^p}^p =0 .
\end{equation}
By help of Lemma \ref{lem4} with $w_1=v$, $w_2=u$, $s=t+\tau_*$, $p=-[t_k]$ and $\tau=-\{t_k\}$, we obtain
\begin{equation*}
\norm{v^b(t+\tau_*-t_k)-u^b(t+\tau_*-\{t_k\})}_{L^p}^p \leq \sum_{j=[t+\tau_*]-1}^{[t+\tau_*]+2}
\norm{V(j-[t_k])-U(j)}_{L^p}^p ,
\end{equation*}
then from \eqref{eq30}, we have
\begin{equation*}
\lim_{k\to+\infty}\norm{v^b(t+\tau_*-t_k)-u^b(t+\tau_*-\{t_k\})}_{L^p}^p =0 .
\end{equation*}
By using $\lim_{k\to+\infty} \{t_k\}=\tau_*$ and  the continuity of $u^b$, we deduce
\begin{equation}
\label{eq35}
\lim_{k\to+\infty}\norm{v^b(t+\tau_*-t_k)-u^b(t)}_{L^p}^p =0 .
\end{equation}
From \eqref{eq34} and \eqref{eq35}, we deduce that $u^b\in AA(\R,L^p(0,1;X))$, then $u\in S_{aa}^p(\R,X)$. 
\vskip2mm
{\bf iii)} For $u\in S_{ap}^p(\R,X)$, we have $Bu=u^b\in AP(\R,L^p(0,1;X))$. 
The restriction of an almost periodic function to the integer numbers is an almost periodic sequence \cite[Theorem 1.27, p. 47]{Co}, 
then $Du=(R\circ B)u\in AP(\Z,L^p(0,1;X))$ where is $R$ the restriction operator defined by \eqref{eq3},
consequently $D\left(S_{ap}^p(\R,X)\right)\subset AP(\Z,L^p(0,1;X))$.
\vskip1mm
For the reciprocal inclusion, let $U\in AP(\Z,L^p(0,1;X))$. By using  i), we can assert the existence  and uniqueness of $u\in BS^p(\R,X)$ such that $Du=U$. Now  prove that
$u\in S_{ap}^p(\R,X)$. 
Let $n_0\in\Z$. From \eqref{eq7}, we have
\begin{equation*}
\sup_{t\in\R} \norm{u^b(t+n_0)-u^b(t)}_{L^p}\leq 2^{\frac{1}{p}}\sup_{n\in\Z} \norm{U(n+n_0)-U(n)}_{L^p} .
\end{equation*}
It follows that  an $\eps$-almost period  of the almost periodic sequence $(U(n))_{n\in\Z}$ is an $\eps_2$-almost period of $u^b$ with $\eps_2=2^{\frac{1}{p}}\eps$,  consequently $u^b\in AP(\R,L^p(0,1;X))$, which gives  $u\in S_{ap}^p(\R,X)$. This ends the proof.
\hfill$\Box$

\vskip 2 mm
From Proposition \ref{th1} we can build a left inverse of the Bochner transform which permits to state that the range of Bochner transform is closed and admits a topological complement. 
For two Banach spaces $E$ and $F$, we denote by $\mathcal{L}(E,F)$  the space of bounded linear operators from $E$ into $F$. Recall that for $B\in\mathcal{L}(E,F)$, a {\it left inverse} of $B$ is an operator $L\in\mathcal{L}(F,E)$ such that $L\circ B = I_E$.

\begin{theorem}
\label{th7}
\vskip2mm
{\bf i)} 
The map $L:BC(\R , L^p(0,1;X))\to BS^p(\R , X)$ defined by
\begin{equation}
\label{eq41} LU(t) = U([t])(\{t\}) , \quad\text{ for } t\in\R 
\end{equation}
is a left inverse of the Bochner transform $B$ from $BS^p(\R , X)$ into $BC(\R , L^p(0,1;X))$. Moreover 
$${\rm Im}(B):=B(BS^p(\R,X)) \text{ and } 
M=\set{V\in BC(\R , L^p(0,1;X))}{\forall n\in\Z ,\, V(n)=0}$$ are two closed subspaces of $BC(\R , L^p(0,1;X))$ and 
$$BC(\R , L^p(0,1;X)) = {\rm Im}(B) \oplus M .$$
\vskip2mm
{\bf ii)} 
The map $L:AA(\R , L^p(0,1;X))\to S_{aa}^p(\R , X)$ defined by \eqref{eq41} 
is a left inverse of the Bochner transform $B$ from $S_{aa}^p(\R , X)$ into $AA(\R , L^p(0,1;X))$. Moreover 
$${\rm Im}(B):=B(S_{aa}^p(\R,X)) \text{ and } 
M=\set{V\in AA(\R , L^p(0,1;X))}{\forall n\in\Z ,\, V(n)=0}$$ are two closed subspaces of $AA(\R , L^p(0,1;X))$ and 
$$AA(\R , L^p(0,1;X)) = {\rm Im}(B) \oplus M .$$
\vskip2mm
{\bf iii)} 
The map $L:AP(\R , L^p(0,1;X))\to S_{ap}^p(\R , X)$ defined by \eqref{eq41} 
is a left inverse of the Bochner transform $B$ from $S_{ap}^p(\R , X)$ into $AP(\R , L^p(0,1;X))$. Moreover 
$${\rm Im}(B):=B(S_{ap}^p(\R,X)) \text{ and } 
M=\set{V\in AP(\R , L^p(0,1;X))}{\forall n\in\Z ,\, V(n)=0}$$ are two closed subspaces of $AP(\R , L^p(0,1;X))$ and 
$$AP(\R , L^p(0,1;X)) = {\rm Im}(B) \oplus M .$$
\end{theorem}

\begin{proof}
\vskip 2 mm
{\bf i)} Consider the restriction operator
$R:BC(\R,L^p(0,1;X))\to \ell^{\infty}(\Z,L^p(0,1;X))$ defined by \eqref{eq3}
and 
the discrete Bochner transform $D:BS^p(\R,X)\to \ell^{\infty}(\Z,L^p(0,1;X))$ defined by \eqref{eq2}.
Obviously $R$ is bounded linear operator and from
Proposition \ref{th1}, $D$ is  an homeomorphism with
$D^{-1}V(t) = V([t])(\{t\})$  for  $t\in\R$. Then 
\begin{equation*}
L=D^{-1} \circ R :BC(\R,L^p(0,1;X))\to BS^p(\R,X)
\end{equation*}
 $$
\xymatrix{
BC(\R,L^p(0,1;X)) \ar[r]^{R} \ar[rd]_{L=D^{-1} \circ R} & \ell^{\infty}(\Z,L^p(0,1;X))   \\
 & BS^p(\R,X)\ar[u]^{D}    \\
}
$$
is a bounded linear operator and $LU(t) = U([t])(\{t\})$ for $t\in\R$. We have $Lu^b(t) = u(t)$ for $u \in BS^p(\R,X)$, then $\ds L\circ B = I_{BS^p(\R,X)}$, which means that $L$
is a left inverse of the Bochner transform $B$ from $BS^p(\R , X)$ into $BC(\R , L^p(0,1;X))$. Then ${\rm Im}(B)$ is closed and
$$BC(\R , L^p(0,1;X)) = {\rm Im}(B) \oplus \ker(L) ,$$
since $L$ is a left inverse of $B$ 
\cite[Theorem 2.13]{Br}. 
We have also $\ker(L) = \ker(R)$ since $L=D^{-1} \circ R$, then
$$BC(\R , L^p(0,1;X)) = {\rm Im}(B) \oplus \set{V\in BC(\R , L^p(0,1;X))}{\forall n\in\Z ,\, V(n)=0} .$$
\vskip 2 mm
{\bf ii)} and {\bf iii)} The proof of ii) (resp. iii)) is similar to i) by considering
the restriction operator
$R:AA(\R,L^p(0,1;X))\to AA(\Z,L^p(0,1;X))$
(resp. $R:AP(\R,L^p(0,1;X))\to AP(\Z,L^p(0,1;X))$) defined by $RU(n) = U(n)$
and 
the discrete Bochner transform $D:S_{aa}^p(\R,X)\to AA(\Z,L^p(0,1;X))$ (resp. $D:S_{ap}^p(\R,X)\to AP(\Z,L^p(0,1;X))$) defined by $Du(n)=u^b(n)$. Then we set $L=D^{-1} \circ R$.
\end{proof}

\section{Characterization of the continuity of Nemytskii operators}
\label{iv}

In this section we give necessary and sufficient conditions to obtain the continuity of Nemytskii operators between almost periodic and almost automorphic spaces in the sense of Stepanov. Then we apply these results to generalize some known results. 
\vskip 2 mm
Both $X$ and $Y$ are Banach spaces. Let $p$ and $q$ be two real numbers in $[1,+\infty)$.
From a map $f:\R\times X\to Y$ we consider the Nemytskii operators of $f$ in the Stepanov almost automorphic case
\begin{equation*}
\mathcal{N}_{f}:S_{aa}^p(\R,X)\to S_{aa}^q(\R,Y)
\end{equation*}
and in the Stepanov almost periodic case
\begin{equation*}
\mathcal{N}_{f}:S_{ap}^p(\R,X)\to S_{ap}^q(\R,Y)
\end{equation*}
defined by 
\begin{equation}
\label{eq4}
\mathcal{N}_{f}(u)(t)=f(t,u(t)) \quad \text{  for } t\in\R .
\end{equation}
We denote by $L^p_1(\R,X)$ the space of $1$-periodic functions in $L_{loc}^p(\R,X)$ 
which is defined by
\begin{equation*}
L^p_1(\R,X)=\set{u\in L_{loc}^p(\R,X)}{u(t+1)=u(t)\text{ for a.e. }t\in\R} .
\end{equation*}
Remark that $L^p_1(\R,X)\subset S_{ap}^p(\R,X)\subset S_{aa}^p(\R,X)$ and for $u\in L^p_1(\R,X)$ we have, 
\begin{equation}
\label{eq44}
\norm{u}_{S^p} = \sup_{t\in\R}\left(\int_0^1\norm{u(t+\theta}^p\,d\theta\right)^{\frac{1}{p}} =\left(\int_0^1\norm{u(\theta}^p\,d\theta\right)^{\frac{1}{p}} ,
\end{equation}
since $u$ is $1$-periodic.

\begin{theorem}
\label{th2}
The following assertions are equivalent.
\vskip2mm
{\bf i)} The Nemytskii operator $\mathcal{N}_{f}$ defined by \eqref{eq4} maps $S_{aa}^p(\R,X)$ into $S_{aa}^q(\R,Y)$ and $\mathcal{N}_{f}$ is continuous.
\vskip2mm
{\bf ii)} The operator $G:L^p_1(\R,X)\to S_{aa}^q(\R,Y)$ defined by  $G(u)=\mathcal{N}_{f}(u)$ for $u\in L^p_1(\R,X)$
is well-defined and continuous.
\vskip2mm
{\bf iii)} For all $x\in X$, $f(\cdot,x)\in S_{aa}^q(\R,Y)$ and
the operator $H:L^p_1(\R,X)\to BS^q(\R,Y)$ defined by  $H(u)=\mathcal{N}_{f}(u)$ for $u\in L^p_1(\R,X)$
is well-defined and continuous.
\end{theorem}

\begin{proof} 
{\bf i) $\Longrightarrow$ ii)} results of the fact that $G$ is the restriction of 
$\mathcal{N}_{f}$ to the subspace $L^p_1(\R,X)$.
\vskip2mm
{\bf ii) $\Longrightarrow$ i)}
Consider the isometry
\begin{equation}
\label{eq37}
J:L^p(0,1;X)\to L^p_1(\R,X) \text{ with } J\omega(t)=\omega(\{t\}) \text{ for } \omega\in L^p(0,1;X) \text{ and } t\in\R 
\end{equation}
and $B$ the Bochner transform between the spaces $S^q_{aa}(\R,Y)$ and $AA(\R,L^q(0,1;Y))$.
Then the map 
$$G_1=B\circ G \circ J: L^p(0,1;X)\to AA(\R,L^q(0,1;Y))$$
$$
\xymatrix{
L^p_{1}(\R,X) \ar[r]^{G} & S^q_{aa}(\R,Y) \ar[d]^{B}  \\
L^p(0,1;X)  \ar[u]^{J} \ar[r]^{G_1} & AA(\R,L^q(0,1;Y))}    \\
$$
 is well-defined and continuous and 
\begin{equation}
\label{eq55}
G_1(\omega)(t)(\theta) = f(t+\theta,\omega(\{t+\theta\}))\quad\text{ for }\omega\in L^p(0,1;X), \,t\in\R \text{ and }\theta\in(0,1). 
 \end{equation}
Consider the following function 
 \begin{equation}
F:\R\times L^p(0,1;X) \to  L^q(0,1;Y) \text{ defined by } F(t,\omega) = G_1(\omega)(t) .
 \end{equation}
 The function $F$ is well-defined and 
$$\underline{F} : L^p(0,1;X)\to AA(\R,L^q(0,1;Y)) \text{ defined by }\underline{F}(\omega) = F(\cdot ,\omega)$$
 is well-defined and continuous since $\underline{F} = G_1$. 
From Theorem \ref{th8}, we can assert that
 the Nemytskii operator $\mathcal{F}$ built on $F$:
\begin{equation}
\label{eq54}
\mathcal{F}:AA(\R,L^p(0,1;X))\to AA(\R,L^q(0,1;Y)) \text{ with }\mathcal{F}(U)(t)=F(t,U(t)) ,
\end{equation}
 is well-defined and  continuous.
 Now we consider $B$ the Bochner transform between the spaces $S^p_{aa}(\R,X)$ and  $AA(\R,L^p(0,1;X))$ and $L$  the left inverse of the Bochner transform between the spaces 
 $AA(\R,L^q(0,1;Y))$ and $S^q_{aa}(\R,Y)$ (cf. Theorem \ref{th7}). 
Then $L \circ \mathcal{F} \circ B : S^p_{aa}(\R,X) \to S^q_{aa}(\R,Y)$
 is well-defined and continuous.
To conclude, we state 
 $$\mathcal{N}_{f} = L \circ \mathcal{F} \circ B :$$
 $$
\xymatrix{
AA(\R,L^p(0,1;X)) \ar[r]^{\mathcal{F}} & AA(\R,L^q(0,1;Y)) \ar[d]^{L}  \\
S^p_{aa}(\R,X)  \ar[u]^{B} \ar[r]^{\mathcal{N}_{f}} & S^q_{aa}(\R,Y)    \\
}
$$
\vskip2mm
For $u\in S^p_{aa}(\R,X)$ one has 
$(L \circ \mathcal{F} \circ B)(u) = (L \circ \mathcal{F})(u^b) = LV$ with  $V=\mathcal{F}(u^b)$.
From \eqref{eq55}-\eqref{eq54}, 
we obtain
$$V(t)(\theta) = \mathcal{F}(u^b)(t)(\theta) = f(t+\theta,u^b(t)(\{t+\theta\}))  = f(t+\theta,u(t+\{t+\theta\})) ,$$
then for $n\in\Z$  we have
$\ds V(n)(\theta)  = f(n+\theta,u(n+\theta))$.

We deduce that 
$\ds LV(t) = V([t])(\{t\}) = f([t]+\{t\},u([t]+\{t\})) = f(t,u(t))$,
then
$\ds (L \circ \mathcal{F} \circ B)(u) (t) = f(t,u(t))$.
 In conclusion $\mathcal{N}_{f} = L \circ \mathcal{F} \circ B$ is is well-defined and continuous.
\vskip2mm
{\bf ii) $\Longrightarrow$ iii)}
The operator $H$ is well-defined and continuous, since $S_{aa}^q(\R,Y)$ is topologically included in $BS^q(\R,Y)$.  
We consider the constant function $u_x:\R\to X$ defined by $u_x(t)=x$ for all $t\in\R$. We have $u_x\in L_{1}^p(\R,X)$,  then $G(u_x) = \mathcal{N}_{f}(u_x)=f(\cdot,x)\in S_{aa}^q(\R,Y)$ for each $x\in X$.
\vskip2mm
{\bf iii) $\Longrightarrow$ ii)} Let $A\subset (0,1)$ be a Lebesgue measurable set  and $\chi_{A}$ its characteristic function. 
Before  we state
\begin{equation}
\label{eq8}
\text{ If } u\in S^q_{aa}(\R,Y) , \text{ then } [t\mapsto u(t)\chi_{A}(\{t\})]\in S^q_{aa}(\R,Y) .
\end{equation}
Let $f :\R\times Y\to Y$ be the function   defined by $f(t,y)=y\chi_{A}(\{t\})$. 
To prove \eqref{eq8}, it suffices to state that the Nemytskii operator $\mathcal{N}_{f}$ defined by $\mathcal{N}_{f}(u)(t) = f(t,u(t)) = u(t)\chi_{A}(\{t\})$
maps $S_{aa}^q(\R,Y)$ into $S_{aa}^q(\R,Y)$.
For $u\in L^p_1(\R,X)$, we obtain that $\mathcal{N}_{f}(u)$  
belongs to $L^p_1(\R,X)$, then  the map $G:L^p_1(\R,X)\to S_{aa}^q(\R,Y)$ with  $G(u)(t)=f(t,u(t)) = u(t)\chi_{A}(\{t\})$ 
is well-defined. Moreover the linear map $G$ is bounded, then $G$ is continuous. From ii) $\Longrightarrow$ i) we obtain that the Nemytskii operator $\mathcal{N}_{f}$ maps $S_{aa}^q(\R,Y)$ into $S_{aa}^q(\R,Y)$.
\vskip 2 mm
To state iii) $\Longrightarrow$ ii), it suffices to state that the range of the function $H$  is in $S^q_{aa}(\R,Y)$, that is
$H(L_1^p(\R,X))\subset S^q_{aa}(\R,Y)$.
Consider the surjective isometry defined by \eqref{eq37}.
Let us denote by $\mathcal{E}(0,1;X)$ the set of simple functions from $(0,1)$ to $X$.  Fix $u\in J(\mathcal{E}(0,1;X))$.
There exists $\omega\in\mathcal{E}(0,1;X)$ such that $u=J\omega$.
The function $\omega$ can be writing as $\ds \omega(\theta)=\sum_{j=1}^{N}x_j\chi_{A_j}(\theta)$ for $\theta\in (0,1)$, where $x_1,\cdots,x_N\in X$ and $\{A_1,\cdots,A_N\}$  is a partition of Lebesgue measurable sets of $(0,1)$.
Then $\ds u(t)=\sum_{j=1}^{N}x_j\chi_{A_j}(\{t\})$ and
it follows 
$$H(u)(t) = f(t,u(t))  = f(t,\sum_{j=1}^{N}x_j\chi_{A_j}(\{t\}))= \sum_{j=1}^{N}f(t,x_j)\chi_{A_j}(\{t\}) .$$
From \eqref{eq8}, we obtain that $f(\cdot,x_j)\chi_{A_j}(\{\cdot\})\in S^q_{aa}(\R,Y)$, since $f(\cdot,x_j)\in S^q_{aa}(\R,Y)$.  Then $H(u)\in S^q_{aa}(\R,Y)$ as finite sum of functions of $S^q_{aa}(\R,Y)$.
We have proved that 
\begin{equation}
\label{eq47}
H(J(\mathcal{E}(0,1;X)))\subset S^q_{aa}(\R,Y) .
\end{equation}
The isometry $J$ is surjective and $\mathcal{E}(0,1;X)$ is dense in $L^p(0,1;X)$, then $J(\mathcal{E}(0,1;X))$ is dense in $L^p_1(\R,X)$. Moreover $H$ is continuous and $S^q_{aa}(\R,Y)$ is closed in $BS^q(\R,Y)$, then from \eqref{eq47}, we deduce that:
$H(L^p_1(\R,X))\subset S^q_{aa}(\R,Y)$.
 \end{proof}

In the proof of Theorem \ref{th2}, by replacing the different almost automorphic spaces  by the corresponding the almost periodic spaces; and by using Theorem \ref{th9} instead Theorem \ref{th8}, we obtain the following result.

\begin{theorem}
\label{th3}
The following assertions  are equivalent.
\vskip2mm
{\bf i)} The Nemytskii operator $\mathcal{N}_{f}$ defined by \eqref{eq4} maps $S_{ap}^p(\R,X)$ into $S_{ap}^q(\R,Y)$ and $\mathcal{N}_{f}$ is continuous.
\vskip2mm
{\bf ii)} The operator $G:L^p_1(\R,X)\to S_{ap}^q(\R,Y)$ defined by  $G(u)=\mathcal{N}_{f}(u)$ for $u\in L^p_1(\R,X)$
is well-defined and continuous.
\vskip2mm
{\bf iii)} For all $x\in X$, $f(\cdot,x)\in S_{ap}^q(\R,Y)$ and
the operator $H:L^p_1(\R,X)\to BS^q(\R,Y)$ defined by  $H(u)=\mathcal{N}_{f}(u)$ for $u\in L^p_1(\R,X)$
is well-defined and continuous.
\end{theorem}

Now we give an application of Theorem \ref{th2} and \ref{th3}, which permits to generalize some known results.

\begin{corollary}
\label{cor8}
Assume that $\ds\frac{1}{q}=\frac{1}{p}+\frac{1}{r}$ for $q$, $p$ and $r\geq1$. Suppose that there exists $L \in BS^r(\R)$ such that
\begin{equation}
\label{eq49}
\norm{f(t,x_1)-f(t,x_2)}\leq L(t) \norm{x_1-x_2}
\end{equation}
 for all $x_1$, $x_2\in X$ and a.e.  $t\in\R$.  Then the following assertions hold.
\vskip 2 mm
{\bf i)} If $f(\cdot,x)\in S_{aa}^q(\R,Y)$ for all $x\in X$, then $\mathcal{N}_{f}$ maps $S_{aa}^p(\R,X)$ into $S_{aa}^q(\R,Y)$ and $\mathcal{N}_{f}$  is continuous.
\vskip 2 mm
{\bf ii)} If $f(\cdot,x)\in S_{ap}^q(\R,Y)$ for all $x\in X$, then $\mathcal{N}_{f}$ maps $S_{ap}^p(\R,X)$ into $S_{ap}^q(\R,Y)$ and $\mathcal{N}_{f}$ is continuous.
\end{corollary}

\begin{remark}
In the framework of metric spaces, Bedouhene et al. have shown that $\mathcal{N}_{f}$ maps $S_{ap}^p(\R,X)$ into $S_{ap}^q(\R,X)$ with assumption \eqref{eq49} when $L \in S_{ap}^r(\R)$  in \cite[Theorem 2.11]{Be-Ch-Me-Ra-Sm}. 
Recall that before this last result of Bedouhene et al.,  similar results assume 
the additional compactness condition {\bf (C1)}: there exists a compact set $K\subset X$ such that $u(t)\in K$ for a.e. $t\in \R$.
More precisely in the almost periodic case, Long et al. has stated the following result in \cite[Theorem 2.2]{Lo-Di}: let $f$ be Stepanov almost periodic in $t\in\R$ uniformly for $x\in X$,
instead of $f(\cdot,x)\in S_{aa}^q(\R,Y)$ for all $x\in X$.
Under assumption \eqref{eq49},
if $u\in S_{ap}^p(\R,X)$ and $u$ satisfies the condition {\bf (C1)}, then $\mathcal{N}_{f}(u)\in S_{ap}^q(\R,X)$. 
For a Stepanov almost automorphic function  in $t\in\R$ uniformly for $x\in X$, a similar result is given by Ding Long et al. in  \cite[Theorem 2.4]{Di-Li-Xi}.
\end{remark} 

\begin{proof}
We proof i) and ii) together. We fix  $u\in BS^p(\R,X)$. The function $f(\cdot,u(\cdot))$  is strongly measurable on each bounded interval of $\R$ since $f(\cdot,x)\in L^q_{loc}(\R,Y)$ for all $x\in X$ and $f(t,\cdot)$ is continuous for a.e. $t\in \R$.
From \eqref{eq49}, we obtain
$$\norm{f(t,x)}\leq L(t) \norm{x}  + \norm{f(t,0)} .$$
From Minkowski's inequality and Holder's inequality with $\frac{r}{q}$ and $\frac{p}{q}$ as exponents, it follows  
$$\sup_{t\in\R}\left(\int^{t+1}_{t}\norm{f(s,u(s))}^q\,ds\right)^{\frac{1}{q}} \leq  \norm{L}_{S^r}  \norm{u}_{S^p} +  \norm{f(\cdot,0)}_{S^q} +   <+\infty ,$$
then the Nemytskii operator $\ds\mathcal{N}_{\phi}$ defined by $\mathcal{N}_{\phi}(u)=\phi(\cdot,u(\cdot))$ maps $BS^p(\R,X)$ into $BS^q(\R,Y)$.
From \eqref{eq49}, we deduce that
the function $\mathcal{N}_{\phi}$ is Lipschitzian with constant $\norm{L}_{S^r}$.
Then  the restriction $H$ of $\mathcal{N}_{\phi}$ to $L_1^p(0,1;X)$,
$H:L_1^p(0,1;X)\to BS^q\R,Y)$ defined by $H(u)(t) = \phi(t,u(t))$
is well-defined and continuous. We conclude by using Theorem \ref{th2} for i) and  Theorem \ref{th3} for ii).
\end{proof}

\section{Some results  in Stepanov spaces}
\label{v}

In this section we generalize some known results of Danilov in \cite{Da, Da1} from Stepanov almost periodic to Stepanov almost automorphic  functions. 
Danilov's demonstrations are not adaptable to spaces of almost automorphic functions.
We give new proofs that deal with almost periodic and almost automorphic cases simultaneously.
Proposition \ref{prop3} and \ref{prop2} will be used in the following sections. 
\vskip 2 mm
Let $X$ be a Banach space and $1\leq p<+\infty$. 
\vskip 2 mm
Let us recall that a subset $\mathcal{H}$ of $L^p(0,1;X)$ is said to be {\it tight} if for every $\eps>0$,  there exists a compact  set $K\subset X$ such that for every   $\omega\in\mathcal{H}$ 
$$\rm{meas}\left(\set{\theta\in(0,1)}{\omega(\theta)\notin K}\right)<\eps  .$$
A subset $\mathcal{H}$ of $L^p(0,1;X)$ is said to be {\it $p$-uniformly integrable} if 
for every $\eps>0$,  there exists $\delta>0$ such that for every $\omega\in\mathcal{H}$ and  for every measurable set $A\subset (0,1)$ with  $\rm{meas}(A)\leq\delta$, we have 
$$\int_A\norm{\omega(\theta)}^p\,d\theta<\eps$$
\vskip 2 mm
A relatively compact subset  $\mathcal{H}$ of $L^p(0,1;X)$ is tight and $p$-uniformly integrable (see e.g. \cite[Corollary 3.3]{Di-Ma}).

\begin{definition}
\label{def2}
\vskip 2mm
{\bf i)} 
A subset $\mathcal{H}$ of $L^p_{loc}(\R,X)$
 is said to be {\it Stepanov tight} if for every $\eps>0$  there exists a compact  set $K\subset X$ such that  
$$\forall u\in \mathcal{H} , \quad\sup_{t\in\R}\Big(\rm{meas}\left(\set{s\in(t,t+1)}{u(s)\notin K}\right)\Big)<\eps  .$$
\vskip 2mm
{\bf ii)} 
A function $u\in L^p_{loc}(\R,X)$  is said to be {\it Stepanov tight} if the set $\{u\}$ is Stepanov tight.
\end{definition}

\begin{definition}
\label{def1}
For $t\in\R$ and $\delta>0$, let us denote by $\mathcal{E}^t_{\delta}$, the class of measurable sets $E\subset (t,t+1)$ such that 
$\rm{meas}(E)\leq\delta$.
\vskip 2mm
{\bf i)} 
A subset $\mathcal{H}$ of $L^p_{loc}(\R,X)$
 is said to be {\it Stepanov $p$-uniformly integrable} if 
$$\lim_{\delta\to0} \left( \sup_{u\in \mathcal{H}}\sup_{t\in\R}\sup_{E\in\mathcal{E}^t_{\delta}}\int_{E}\norm{u(s)}^p\,ds \right) =0 .$$
\vskip 2mm
{\bf ii)} 
A function $u\in L^p_{loc}(\R,X)$  is said to be {\it Stepanov $p$-uniformly integrable} if the set $\{u\}$ is Stepanov $p$-uniformly integrable.
\end{definition}

\vskip 2 mm
By using the Bochner transform the tightly and the $p$-uniformly integrability in the sense of Stepanov can be reduced to the classical notion of tightly and the $p$-uniformly integrability of 
a subset of  $L^p(0,1;X)$.

\begin{lemma}
\label{prop5}
Let $\mathcal{H}$ be a subset of $L^p_{loc}(\R,X)$. 
Denotes by $\mathcal{H}_b$ the following subset of $L^p(0,1;X)$, $\mathcal{H}_b = \set{u^b(t)}{ u\in\mathcal{H} \text{ and } t\in\R }$.
\vskip 2mm
{\bf i)} $\mathcal{H}$ is Stepanov tight if and only if $\mathcal{H}_b$  is tight in $L^p(0,1;X)$.
\vskip 2mm
{\bf ii)} $\mathcal{H}$  is Stepanov $p$-uniformly integrable if and only if $\mathcal{H}_b$  if $p$-uniformly integrable in $L^p(0,1;X)$.
\end{lemma}

\begin{proof}
{\bf i)} We have $\set{s\in(t,t+1)}{u(s)\notin K} =t+\set{\theta\in(0,1)}{u^b(t)(\theta)\notin K}$ and by using the invariance by translation of the Lebesgue's measure on $\R$ we obtain
$$\rm{meas}\left(\set{s\in(t,t+1)}{u(s)\notin K}\right)=\rm{meas}\left(\set{\theta\in(0,1)}{u^b(t)(\theta)\notin K}\right) .$$
From this last equality, we deduce i).
\vskip 2 mm
{\bf ii)} For $E\in\mathcal{E}^t_{\delta}$, we have $\ds\int_{E}\norm{u(s)}^p\,ds = \int_{E-t}\norm{u^b(t)(\theta)}^p\,d\theta$ and by using $\mathcal{E}^t_{\delta} = t+\mathcal{E}^0_{\delta}
$, we obtain
$$\sup_{E\in\mathcal{E}^t_{\delta}}\int_{E}\norm{u(s)}^p\,ds = \sup_{E\in\mathcal{E}^0_{\delta}}\int_{E}\norm{u^b(t)(\theta)}^p\,d\theta .$$
From this last equality, we deduce ii).
\end{proof}

\begin{proposition}
\label{prop3}
If $\mathcal{K}$ is a compact set in $S^p_{aa}(\R,X)$, then $\mathcal{K}$ is 
Stepanov tight and Stepanov $p$-uniformly integrable. 
\end{proposition}

Before proving Proposition \ref{prop3}, we make the following remarks.

\begin{remark}
\label{rq3}
{\bf i)} Proposition \ref{prop3} also holds  if $\mathcal{K}$ is a compact set in $S^p_{ap}(\R,X)$.
\vskip 2 mm
{\bf ii)} The assertion "$u$ is Stepanov tight when $u\in S^p_{ap}(\R,X)$"
is contained in a more general result of Danilov \cite[Theorem 3]{Da1}, which assert that if $u\in S^p_{ap}(\R,X)$, then there exist $v\in AP(\R,X)$ and a measurable set $N\subset \R$ such that $u(t)=v(t)$ for $t\notin N$ and
$\ds\sup_{t\in\R}\Big(\rm{meas}\left((t,t+1)\cap N\right)\Big)<\eps$.
\vskip 2 mm
{\bf iii)} The assertion "$u$ is Stepanov $p$-uniformly integrable when $u\in S^p_{ap}(\R,X)$"
 is contained in a more general result of Danilov \cite[page 1420]{Da}, which gives a characterization of the space $S^p_{ap}(\R,X)$ in term of the so-called {\it Stepanov almost periodicity in Lebesgue measure} and Stepanov $p$-uniformly integrability.
 This characterization is the following: $u\in S^p_{ap}(\R,X)$ if and only is $u$ is Stepanov $p$-uniformly integrable and $u$ Stepanov almost periodicity in Lebesgue measure, that is for any $\eps>0$ and $\delta>0$ the set
 $$\set{\tau\in\R}{\sup_{t \in \mathbb{R}}\rm{meas}\left( \set{s\in[t,t+1]}{\norm{(u (s + \tau) - u (s)}_{X}\geq  \eps}\right)<\delta}$$
 is relatively dense in $\R$. 
\end{remark}

To prove  Proposition \ref{prop3} we use the following lemma.

\begin{lemma}
\label{lem1}
Let $E$ be a Banach space. If $\ds\mathcal{K}$ is a compact subset of $AA(\R,E)$, then $\ds S=\set{U(t)}{U\in\mathcal{K}\text{ and }t\in\R}$ is a relatively compact in $E$. 
\end{lemma}

\begin{remark}
In \cite[Lemma 3.6]{Bl-Ci-Ng-Pe}, Blot et al.,  Lemma \ref{lem1} is stated in the case of almost periodic spaces in the sense of Bohr.
\end{remark}

\begin{proof}
It suffices to prove that all sequence $(U_k(t_k))_{k\in\N}$ as at least a cluster point $x$ of $E$, where $U_k\in\mathcal{K}$ and $t_k\in\R$. The sequence $(U_k)_{k\in\N}$ as at least a cluster point $U\in\mathcal{K}$, since $\mathcal{K}$ is a compact subset of $AA(\R,E)$. The sequence $(U(t_k)_{k\in\N}$ as at least a cluster point $x\in E$, since the range of an almost automorphic function is relatively compact \cite[Theorem 2.3, p. 11]{NG}.
Consequently, there exists a common subsequence of $(U_k)_{k\in\N}$ and $(t_k)_{k\in\N}$  that we note in the same way such that
$$
\lim_{k\to+\infty}\sup_{t\in\R}\norm{U_k(t)-U(t)}_E=0 \quad\text{and}\quad \lim_{k\to+\infty}\norm{U(t_k)-x}_E=0 .
$$ 
To ends, we use the following inequality
$$\norm{U_k(t_k)-x}_E \leq \norm{U_k(t_k)-U(t_k)}_E + \norm{U(t_k)-x}_E .$$
\end{proof}

\vskip    2mm \noindent\textit{Proof of Proposition \ref{prop3}.}
The range $\set{u^b}{u\in\mathcal{K}}$ of the compact $\mathcal{K}$
by the Bochner transform is a compact subset of $AA(\R,L^p(0;1;X)))$, since the Bochner transform is an isometry from $S^p_{aa}(\R,X)$ into $AA(\R,L^p(0;1;X)))$. 
Denote by $\ds\mathcal{K}_b$ the following subset of $L^p(0;1;X))$
\begin{equation*}
\mathcal{K}_b = \set{u^b(t)}{u\in\mathcal{K} \text{ and } t\in\R} .
\end{equation*}
By help of Lemma \ref{lem1}, we deduce that $\ds\mathcal{K}_b$ is a relatively compact in $L^p(0;1;X))$.
It follows that $\mathcal{K}_b$ is tight and $p$-uniformly integrable \cite[Corollary 3.3]{Di-Ma}.
The conclusion results of Lemma \ref{prop5}
\hfill$\Box$ 

\begin{proposition}
\label{prop2}
Assume  that $u$ and $u_k\in BS^p(\R,X)$ for all $k\in\N$. If 
$\ds\set{u_k}{k\in\N}$ is  Stepanov $p$-uniformly integrable,
then $\ds \lim_{k\to+\infty}\norm{u_k-u}_{S^p}=0$ if and only if 
\begin{equation}
\label{eq13}
\forall\eps>0, \qquad \lim_{k\to\infty}\sup_{t\in\R}\Big(
\rm{meas}\left( \set{s\in (t,t+1)}{\norm{u_k(s)-u(s)}\geq\eps}\right) \Big)=0 .
\end{equation}
\end{proposition}

\begin{proof}
{\bf i)} $\Longrightarrow$ results of the following Tchebychev's inequality
\begin{equation*}
\rm{meas}\left( \set{s\in (t,t+1)}{\norm{u_k(s)-u(s)}\geq\eps}\right) 
\leq\frac{1}{\eps^p}\int_t^{t+1}\norm{u_k(s)-u(s)}^p\,ds .
\end{equation*}
\vskip 2mm
{\bf ii)} $\Longleftarrow$
\vskip 2 mm
{\it Step 1: we assume that $u=0$.} We fix $\eps>0$. 
We want to show that 
\begin{equation}
\label{eq15}
\exists k_0\in\N ,\quad k\geq k_0\implies\sup_{t\in\R}\int_{t}^{t+1} \norm{u_k(s)}^p \,ds\leq\eps .
\end{equation}
By hypothesis, the subset $\ds\set{u_k}{k\in\N}$ of $BS^p(\R,X)$ is  Stepanov $p$-uniformly integrable, then
there exists $\delta>0$ such that
\begin{equation}
\label{eq36}
\sup_{k\in\N}\sup_{t\in\R}\sup_{E\in\mathcal{E}^t_{\delta}}\int_{E}\norm{u_k(s)}^p\,ds \leq \frac{\eps}{2} .
\end{equation}
Let us denote by 
$$A^t_k=\set{s\in(t,t+1)}{\norm{u_k(s)}\geq\left(\frac{\eps}{2}\right)^{\frac{1}{p}}} .$$
From Hypothesis \eqref{eq13}, we have
\begin{equation}
\label{eq22}
\exists k_0\in\N ,\quad k\geq k_0\implies\sup_{t\in\R}\Big(
\rm{meas}\left( A^t_k\right) \Big) \leq\delta .
\end{equation}
From  inequalities
$$\int_{t}^{t+1} \norm{u_k(s)}^p \,ds \leq \int_{(t,t+1)\setminus A^t_k} \norm{u_k(s)}^p \,ds + \int_{A^t_k} \norm{u_k(s)}^p \,ds \leq \frac{\eps}{2} + \int_{A^t_k} \norm{u_k(s)}^p \,ds ,$$
and from \eqref{eq36} and \eqref{eq22}, we obtain \eqref{eq15}, then  the claim is proved.
\vskip 2 mm
{\it Step 2: general case.}
\vskip 2 mm
{\it First we prove $\ds\eqref{eq13} \Longrightarrow$ $u$ is Stepanov $p$-uniformly integrable.}
We fix $t\in\R$ and $E$ a measurable set of $(t,t+1)$. By assumption \eqref{eq13}, we deduce that the sequence $(u_{k})_{k\in\N}$ tends to $u$ in measure on $(t,t+1)$, that is 
\begin{equation*}
\forall\eps>0, \qquad \lim_{k\to\infty}
\rm{meas}\left( \set{s\in (t,t+1)}{\norm{u_k(s)-u(s)}\geq\eps}\right) =0 ,
\end{equation*}
 then there exists a subsequence $(u_{\alpha(k)})_{k\in\N}$ of $(u_{k})_{k\in\N}$ such that
$\ds\lim_{k\to +\infty} u_{\alpha(k)}(s) = u(s)$ for a.e. $s\in (t,t+1)$ (cf. \cite[Theorem 3, p. 45]{Di-Uh}).
By using Fatou's lemma, we have
$$\int_E \norm{u(s)}^p \,ds = \int_E \lim_{k\to +\infty} \norm{u_{\alpha(k)}(s)}^p \,ds \leq \liminf_{k\to +\infty}\int_E \norm{u_{\alpha(k)}(s)}^p \,ds  ,$$
then
$$\int_E \norm{u(s)}^p \,ds \leq \sup_{k\in\R}\int_{E}\norm{u_{\alpha(k)}(s)}^p\,ds .$$
We deduce that $\{u\}$ is Stepanov $p$-uniformly integrable, since by hypothesis $\ds\set{u_k}{k\in\N}$ is Stepanov $p$-uniformly integrable.
\vskip 2 mm
{\it Secondly we prove $\ds\eqref{eq13} \Longrightarrow \lim_{k\to+\infty}\norm{u_k-u}_{S^p}=0$.}
Let us denote by $v_k=u_k-u$. From the inequality
$\ds\norm{v_k(t)} \leq \norm{u_k(t)} + \norm{u(t)}$, 
we deduce that $\ds\set{v_k}{k\in\N}$ is Stepanov $p$-uniformly integrable, since $\{u\}$ and $\ds\set{u_k}{k\in\N}$ are  Stepanov $p$-uniformly integrable. We conclude by using step 1 on the sequence $(v_k)_{k\in\N}$.
\end{proof}

\vskip 3 mm
The following corollary extends a result of Danilov \cite[Lemma 1]{Da} from the almost periodic case case to the almost automorphic case case

\begin{corollary}
\label{prop1}
Suppose that $u\in S^p_{aa}(\R,X)$ and $u_k\in S^p_{aa}(\R,X)$ for all $k\in\N$. Then $\ds \lim_{k\to+\infty}\norm{u_k-u}_{S^p}=0$ if and only if $\ds\set{u_k}{k\in\N}$ is  Stepanov $p$-uniformly integrable and \eqref{eq13} holds.
\end{corollary}

\begin{proof}
It is a consequence of Proposition \ref{prop3} and \ref{prop2}, since $\ds\set{u_k}{k\in\N}\cup\{u\}$ is a compact subset of $S^p_{aa}(\R,X)$, when $u_k\to u$ in $S^p_{aa}(\R,X)$ as $k\to+\infty$.
\end{proof}

\section{Sufficient conditions for the continuity of Nemytskii operators}
\label{vi}

Both $X$ and $Y$ are Banach spaces. Let $p$ and $q$ be two real numbers in $[1,+\infty)$.
From a map $f:\R\times X\to Y$ we give sufficient conditions to obtain the continuity of the Nemytskii operators built on $f$ in the Stepanov almost automorphic case
\begin{equation*}
\mathcal{N}_{f}:S_{aa}^p(\R,X)\to S_{aa}^q(\R,Y)
\end{equation*}
and in the Stepanov almost periodic case
\begin{equation*}
\mathcal{N}_{f}:S_{ap}^p(\R,X)\to S_{ap}^q(\R,Y)
\end{equation*}
defined by 
\begin{equation}
\label{eq45}
\mathcal{N}_{f}(u)(t)=f(t,u(t)) \quad \text{  for } t\in\R .
\end{equation}
In Section \ref{iv}, we have given necessary and sufficient conditions to obtain the continuity of Nemytskii operators between Stepanov spaces. 
The assumptions used are on some Nemytskii operators built on $f$ (cf. theorems \ref{th2} and \ref{th3}). 
In this section, the assumptions will be directly on the function $f$.

\subsection{Statement of results}

Now we formulate the following hypotheses.
\vskip2mm
{\bf (H1)} {\it There exist a constant $a>0$ and $b \in S^q_{aa}(\R)$ such that
$\ds\norm{f(t,x)}\leq a \norm{x}^{\frac{p}{q}} + b(t)$ for all $x\in X$ and 
a.e. $t\in\R$.}
\vskip2mm
{\bf (H2)} {\it For a.e. $t\in \R$, the map $f(t,\cdot)$ is continuous from $X$ into $Y$.}
\vskip 2mm
{\bf (H3)} {\it For all compact  set $K\subset X$, for all $r>0$, there exists a measurable set $N\subset \R$ such that $\ds\sup_{t\in\R}\Big(\rm{meas}\left((t,t+1)\cap N\right)\Big)<r$ and
$$\forall \eps>0,\, \exists \delta>0, \, \forall x_1 ,  x_2\in K, \, \forall t\in\R\setminus N, \quad\norm{x_1-x_2}\leq \delta \implies \norm{f(t,x_1)-f(t,x_2)}\leq\eps .
$$}

\begin{remark} 
{\it About Hypothesis {\bf (H1)}.} In the context of separable Banach spaces and in the periodic case: $f(t+T,x)=f(t,x)$, we will see in Corollary \ref{cor5}, that an equivalent hypothesis to {\bf (H1)} in the periodic case which is denoted {\bf (H4)}, is a necessary condition for that the Nemytskii operator $\mathcal{N}_{f}$ maps $S_{aa}^p(\R,X)$ into $S_{aa}^q(\R,Y)$ (resp. $S_{ap}^p(\R,X)$ into $S_{ap}^q(\R,Y)$). Hypothesis  {\bf (H4)} is {\bf (H1)} where $b \in L^q(0,T)$ and the inequality  holds for $t\in(0,T)$ (cf. page \pageref{pa1}).
\end{remark}

\begin{remark} {\it About  Hypothesis {\bf (H2)}.}
 To state that the Nemytskii operator $\mathcal{N}_{f}$ maps $S_{aa}^p(\R,X)$ into $S_{aa}^q(\R,Y)$ or  $S_{ap}^p(\R,X)$ into $S_{ap}^q(\R,Y)$, a necessary condition is that the function $f(\cdot,u(\cdot))$ is strongly measurable on each bounded interval, for $u\in S^p_{ap}(\R,X)\subset S^p_{aa}(\R,X)$. Without Hypothesis {\bf (H2)}, it is difficult to reach this necessary condition.
\end{remark}

\begin{remark} {\it About  Hypothesis {\bf (H3)}.}
\label{rq2}
{\bf i)}
For a compact $K\subset X$ and $\delta>0$, let us denote by
\begin{equation}
\label{eq57}
\alpha^K_{\delta}(t)=\sup\set{\norm{f(t,x_1)-f(t,x_2}}{x_1\in K,\,x_2\in K,\, \norm{x_1-x_2}\leq\delta} .
\end{equation}
Then  Hypothesis {\bf (H3)} is equivalent to the following assertion: for all compact  set $K\subset X$ and for all $r>0$, there exists  a measurable set $N\subset \R$ such that 
$$\sup_{t\in\R}\Big(\rm{meas}\left((t,t+1)\cap N\right)\Big)<r \text{ and } \alpha^K_{\delta}(t)\to0 \text{ uniformly on } \R\setminus N \text{ as }\delta\to 0 .$$
Remark that $\alpha^K_{\delta}(t)\to0$ for a.e.  $t\in \R$ as $\delta\to 0$, since $f(t,\cdot)$ is uniformly continuous on the compact set $K$, but Hypothesis {\bf (H3)} is not necessarily satisfied.
\vskip 2 mm
{\bf ii)} Formulate the following condition:
\vskip 2 mm
{\bf (C2)} {\it  There exist $a\in BS^1(\R)$ and $\ds\eps:\R\to\R$ with $\ds\lim_{\delta\to0}\eps(\delta)=0$ such that 
$$\alpha^K_{\delta}(t) \leq a(t)\eps(\delta)$$
where $\alpha^K_{\delta}$ is defined by \eqref{eq57}.}
\vskip 2 mm
Let us denote by $\ds N_R=\set{t\in\R}{a(t)>R}$ for $R>0$. 
From the Tchebychev's inequality, we deduce that
$\norm{a}_{S^1}\geq R\,  \rm{meas}\left((t,t+1)\cap N_R\right)$, then $\ds\sup_{t\in\R}\Big(\rm{meas}\left((t,t+1)\cap N_R\right)\Big)\to0$ as $R\to+\infty$. 
We also have $\ds\lim_{\delta\to0}\alpha^K_{\delta}(t)\to0$  uniformly in $t\in\R\setminus N_R$.
Consequently Condition {\bf (C2)} implies Hypothesis {\bf (H3)}.
\vskip 2 mm
Evidently Condition {\bf (C2)} holds if  the assumption \eqref{eq49} 
of Corollary \ref{cor8} holds.
\end{remark}

\vskip 2 mm
Now we give two theorems which are equivalent.

\begin{theorem}
\label{th5}
Suppose that {\bf (H1)-(H3)} hold. 
\vskip 2 mm
{\bf i)} If $f(\cdot,x)\in S_{aa}^q(\R,Y)$ for all $x\in X$, then the Nemytskii operator $\mathcal{N}_{f}$ defined by \eqref{eq45} maps $S_{aa}^p(\R,X)$ into $S_{aa}^q(\R,Y)$ and $\mathcal{N}_{f}$ is continuous.
\vskip 2 mm
{\bf ii)} If $f(\cdot,x)\in S_{ap}^q(\R,Y)$ for all $x\in X$, then $\mathcal{N}_{f}$ maps $S_{ap}^p(\R,X)$ into $S_{ap}^q(\R,Y)$ and $\mathcal{N}_{f}$ is continuous.
\end{theorem}

The proof of Theorem \ref{th5} is given in Subsection \ref{vi-ii}

\begin{remark}
Theorem \ref{th5} holds if $f$ satisfies Hypothesis {\bf (H1)} and  $f\in AA_U(\R\times X,Y)$ for i) and $f\in AP_U(\R\times X,Y)$) for ii) (cf. resp. Theorem \ref{th8} and \ref{th9}).  
\end{remark}

From a function  $f:\R\times X\to Y$ and $K\subset X$ a compact set, we consider  $\tilde{f}_K$ the map defined by
\begin{equation}
\label{eq63}
\tilde{f}_K : \R \to C(K,Y)) \text{ with }\tilde{f}_K(t)=f(t,\cdot) \text{ for a.e. }t\in\R .\end{equation}

\begin{theorem}
\label{th11}
Suppose that {\bf (H1)} holds. 
\vskip 2 mm
{\bf i)} If for all compact set $K\subset X$, $\tilde{f}_K\in S_{aa}^q(\R,C(K,Y))$, then $\mathcal{N}_{f}$ maps $S_{aa}^p(\R,X)$ into $S_{aa}^q(\R,Y)$ and $\mathcal{N}_{f}$ is continuous.
\vskip 2 mm
{\bf ii)} If for all compact set $K\subset X$, $\tilde{f}_K\in S_{ap}^q(\R,C(K,Y))$, then $\mathcal{N}_{f}$ maps $S_{ap}^p(\R,X)$ into $S_{ap}^q(\R,Y)$ and $\mathcal{N}_{f}$ is continuous.
\end{theorem}

The proof of Theorem \ref{th11} is given in Subsection \ref{vi-ii}. We will prove that the hypotheses of Theorem \ref{th5} are equivalent to those of Theorem \ref{th11}.

\begin{remark}
\label{rq4}
In  \cite[Lemma 3]{Da}, Danilov has stated  in the context of metric spaces,  that if $u:\R\to X$ and $\tilde{f}:\R\to BC(X,Y)$ with $\tilde{f}(t)=f(t,\cdot)$ are Stepanov almost periodic in Lebesgue measure (see the definition in Remark \ref{rq3}, iii)), then $\mathcal{N}_{f}(u):\R\to Y$ is also Stepanov almost periodic in Lebesgue measure. Then by using this last result of Danilov, in the context of Banach space, Andres et al. \cite[Proposition 3.4]{An-Pe} have shown that $\mathcal{N}_{f}$ maps $S_{ap}^p(\R,X)$ into $S_{ap}^q(\R,X)$, when $\tilde{f}\in S_{ap}^q(\R,BC(X,X))$. 
In the almost periodic case, Theorem \ref{th11} is an improvement of \cite[Proposition 3.4]{An-Pe}.
Even in the linear case \cite[Proposition 3.4]{An-Pe} 
does not allow us to conclude.
 Indeed if 
$f(t,x)=A(t)x$ with $A\in S_{ap}^r(\R,\mathcal{L}(X))$ for $r\geq1$, where $\mathcal{L}(X))$ stands for the set of bounded and linear maps from $X$ in to itself,
assumptions of \cite[Proposition 3.4]{An-Pe} are not satisfied, since $\mathcal{L}(X)) \nsubseteq BC(X,X)$. But Theorem \ref{th11} permits to conclude for $p$ and $q\geq1$ such that $\ds\frac{1}{q}=\frac{1}{p}+\frac{1}{r}$.
\end{remark}

\vskip 2 mm
To apply Theorem \ref{th11} it is necessary to establish that $\tilde{f}_K\in L_{loc}^q(\R,C(K,Y))$,  especially that $\tilde{f}_K$ is strongly measurable on each bounded interval, for that we give the following criterion which will be used to prove Theorem \ref{th11}.

\begin{lemma}
\label{lem6}
Suppose that {\bf (H1)-(H2)} hold and $f(\cdot,x)\in L_{loc}^q(\R,Y)$ for all $x\in X$. 
Then  $\tilde{f}_K\in L_{loc}^q(\R,C(K,Y))$ for all compact set $K\subset X$.
\end{lemma}

\begin{proof} 
Let $K$ be a compact subset of $X$. Consider the restriction of $f$ on $(a,b)\times K$.
For each $x\in K$, the function $f(\cdot,x)$ is strongly measurable on   $(a,b)$, then $f(\cdot,x)$ is essentially separably valued, i.e., there exists a subset $N_x \subset (a,b)$ of measure zero such that the set $\set{f(t,x)}{t\in(a,b)\setminus N_x}$ is separable \cite[Lemma 11.36, p. 417]{AB}. A a compact set is separable, then there exists a subset $D\subset K$ that is countable and dense in $K$.
If we denote by $\ds N=\bigcup_{x\in D} N_x$, then its  measure is null and the set
$A_x=\set{f(t,x)}{t\in(a,b)\setminus N}$ is separable. It follows that 
$\ds\overline{\bigcup_{x\in D}A_x}$ is separable. 
 The function $f(t,\cdot)$ being continuous a.e. $t\in(a,b)$, we deduce that 
 $\ds\bigcup_{x\in K}A_x \subset\overline{\bigcup_{x\in D}A_x}$, it follows that $\ds\bigcup_{x\in K}A_x$ is separable.
Then there exists a separable closed subspace $Y_s \subset Y$ such that
the values $f(t,x)$ lie  in a separable closed subspace $Y_s$ of $Y$ for a.e. $t\in(a,b)$ and for all $x\in K$. We can consider the restriction of $f$ on $Y_s$: 
$f:(a,b)\times K \to Y_s$.
Since $Y_s$ is separable,  the notions of measurability and strong measurability  are equivalent \cite[Lemma 11.36, p. 417]{AB}. The function $f:(a,b)\times K \to Y_s$ is a Carath\'edory function: $f(\cdot,x)$ is measurable on $(a,b)$ for all $x\in K$ and $f(t,\cdot)$ is continuous from $K$ into $Y_s$ for a.e. $t\in (a,b)$. Then $\tilde{f}_K$ maps $(a,b)$ into $C(K,Y_s)$ and is Borel measurable, since $K$ is a compact metric space and $Y_s$ a separable Banach space (cf. \cite[Theorem 4.54, p. 153]{AB}). The Banach space $C(K,Y_s)$ is separable, since $Y_s$ is separable \cite[Lemma 3.85, p. 120]{AB}, then $\tilde{f}_K$  is 
strong measurable on $(a,b)$. 
From {\bf (H1)}, we have 
\begin{equation*}
\forall t\in\R ,\qquad \norm{\tilde{f}_K(t)}_{C(K,Y)} \leq a \sup_{x\in K}\norm{x}^{\frac{p}{q}} + b(t)
\end{equation*}
with $\ds a \sup_{x\in K}\norm{x}^{\frac{p}{q}} + b(\cdot) \in S^q_{aa}(\R)\subset L^q_{loc}(\R)$, then
$\norm{\tilde{f}_K(\cdot)}_{C(K,Y_s)}\in L^q_{loc}(\R)$.
The function $\tilde{f}_K$ being strongly measurable on each bounded interval of $\R$ and $\norm{\tilde{f}_K(\cdot)}_{C(K,Y_s)}\in L^q_{loc}(\R)$, we have $\tilde{f}_K\in L^q_{loc}(\R,C(K,Y_s))$
\cite[Theorem 11.43, p. 420]{AB}.
By considering the standard isometry from $C(K,Y_s)$ to $C(K,Y)$, we obtain the result: 
$\tilde{f}_K\in L^q_{loc}(\R,C(K,Y))$.
\end{proof} 

\vskip 2 mm
In the periodic case, we can prove that Hypothesis {\bf (H3)} 
holds under  {\bf (H2)}, but we will not use this remark. 
For $T>0$, consider the subspace $L^p_T(\R,X)$ of $S_{ap}^p(\R,X)$ defined by
\begin{equation*}
L^p_T(\R,X)=\set{u\in L_{loc}^p(\R,X)}{u(t+T)=u(t)\text{ for a.e. }t\in\R} .
\end{equation*}

\vskip2mm
For a map $f:\R\times X\to Y$,  we formulate the following hypotheses.
\vskip2mm
\label{pa1}
{\bf (H4)} {\it There exist $a>0$ and $b\in L^q(0,T)$ such that $\norm{f(s,x)}\leq a\norm{x}^{\frac{p}{q}}+b(s)$, for all $x\in X$ and a.e.  $s\in (0,T)$. 
}
\vskip2mm
{\bf (H5)} {\it For all $x\in X$, $f(\cdot,x)\in L^q_{T}(\R,Y)$.}
\vskip2mm

\begin{corollary}
\label{cor3}
We assume that {\bf (H2), (H4)} and {\bf (H5)} hold.
Then the following assertions hold.
\vskip 2 mm
{\bf i)} The Nemytskii operator $\mathcal{N}_{f}$ maps $S_{aa}^p(\R,X)$ into $S_{aa}^q(\R,Y)$ and $\mathcal{N}_{f}$ is continuous.
\vskip 2 mm
{\bf ii)} The Nemytskii operator $\mathcal{N}_{f}$ maps $S_{ap}^p(\R,X)$ into $S_{ap}^q(\R,Y)$ and $\mathcal{N}_{f}$ is continuous.
\end{corollary}

\begin{proof}
We proof i) and ii) together. 
For that we use Theorem \ref{th11}. 
By using {\bf (H4)} and the $T$-periodicity of the function $f(\cdot,x)$, Hypothesis {\bf (H1)} holds with the function $\ds t\mapsto b\left(\left\{\frac{t}{T}\right\}T\right)\in L_{T}^q(\R,Y)\subset S_{aa}^q(\R,Y)$. From {\bf (H5)}, we have $f(\cdot,x)\in L^q_{T}(\R,Y) \subset L_{loc}^q(\R,Y)$ for all $x\in X$, then by Lemma \ref{lem6}, we have $\tilde{f}_K\in L_{loc}^q(\R,C(K,Y))$ for all compact set $K\subset X$. 
Moreover by Hypothesis {\bf (H5)}, we have $\tilde{f}_K(t+T) = \tilde{f}_K(t)$ for all $t\in\R$: $\tilde{f}_K\in L_{T}^q(\R,C(K,Y)) \subset S_{ap}^q(\R,C(K,Y)) \subset S_{aa}^q(\R,C(K,Y))$. The conclusion results of Theorem \ref{th11}.
\end{proof}

Now we give a direct consequence in the autonomous case of Corollary \ref{cor3}.

\begin{corollary}
\label{cor10}
Suppose that $f:X\to Y$ is a continuous map and there exist $a>0$ and $b>0$  such that
\begin{equation}
\label{eq50}
\forall x\in X ,\quad\norm{f(x)}\leq a\norm{x}^{\frac{p}{q}}+b .
\end{equation}
\ We consider the Nemytskii operator $\mathcal{N}_{f}$ defined by
$\mathcal{N}_{f}(u)=f\circ u$. Then the following assertions hold.
\vskip 2 mm
{\bf i)} The  Nemytskii operator $\mathcal{N}_{f}$ maps $S_{aa}^p(\R,X)$ into $S_{aa}^q(\R,Y)$  and $\mathcal{N}_{f}$ is continuous.
\vskip 2 mm
{\bf ii)} The  Nemytskii operator $\mathcal{N}_{f}$ maps $S_{ap}^p(\R,X)$ into $S_{ap}^q(\R,Y)$  and $\mathcal{N}_{f}$ is continuous.
\end{corollary}

\begin{remark}
\label{rq5}
With the same assumptions as those of Corollary \ref{cor10}, Andres et al. \cite[Lemma 3.2]{An-Pe} have shown that $\mathcal{N}_{f}$ maps $S_{ap}^p(\R,X)$ into $S_{ap}^q(\R,Y)$. We will see in Corollary \ref{cor5} that in reflexive spaces, assumption \eqref{eq50} is also a necessary condition.
\end{remark}

\vskip 2 mm
When the Banach spaces $X$ and $Y$ are separable, we can improve Corollary \ref{cor3}.

\begin{corollary}
\label{cor5}
We assume that  the Banach spaces $X$ and $Y$ are separable. 
We assume that {\bf (H2)} and {\bf (H5)} hold.
Then the following assertions hold.
\vskip 2 mm
{\bf i)}
The Nemytskii operator $\mathcal{N}_{f}$ maps $S_{aa}^p(\R,X)$ into $S_{aa}^q(\R,Y)$  if and only if {\bf (H4)} holds.
In this case $\mathcal{N}_{f}$ is continuous.
\vskip 2 mm
{\bf ii)}
The Nemytskii operator $\mathcal{N}_{f}$ maps $S_{ap}^p(\R,X)$ into $S_{ap}^q(\R,Y)$  if and only if {\bf (H4)} holds.
In this case $\mathcal{N}_{f}$ is continuous.
\end{corollary}

\begin{proof} 
We proof i) and ii) together. When $E$ is a Banach space, $S^p(\R,E)$ denotes 
indifferently
$S^p_{aa}(\R,E)$ or $S^p_{ap}(\R,E)$. 
By using Corollary \ref{cor3}, it suffices to state that
\begin{equation}
\label{eq62}
\text{if } \mathcal{N}_{f} \text{ defined by \eqref{eq45} maps } S^p(\R,X) \text{ into }S^q(\R,Y)\text{, then  {\bf (H4)} holds.}
\end{equation}
Consider the restriction map  $H:L^p_T(\R,X)\to BS^q(\R,Y)$ of $\mathcal{N}_{f}$ defined by $H(u)=\mathcal{N}_{f}(u)$ for $u\in L^p_T(\R,X)$.  
Remark that $H(u) = \mathcal{N}_{f}(u) = f(\cdot , u(\cdot))$ is $T$-periodic  for $u\in L^p_T(\R,X)$, then $\ds H(L^p_T(\R,X)) \subset L^q_T(\R,Y)$. Then the map $H_1:L^p_T(\R,X)\to L^q_T(\R,Y)$ with $H_1(u) = H(u)$ for $u\in L^p_T(\R,X)$ is well-defined. 
Consider the following isometries
\begin{equation*}
J_p:L^p(0,T;X)\to L^p_T(\R,X) \text{ with } J_p\omega(t)=\omega\left(\left\{\frac{t}{T}\right\}T\right) \text{ for } \omega\in L^p(0,1;X) \text{ and } t\in\R ,
\end{equation*}
\begin{equation*}
J_q:L^q(0,T;Y)\to L^q_T(\R,Y) \text{ with } J_p\omega(t)=\omega\left(\left\{\frac{t}{T}\right\}T\right) \text{ for } \omega\in L^q(0,1;Y) \text{ and } t\in\R .
\end{equation*}
Obviously $J_q$ is surjective and $\ds (J_q)^{-1}u(s) = u(s)$ for $s\in(0,T)$.
Then the map 
$$g= (J_q)^{-1} \circ H_1 \circ J_p:L^p(0,1;X)\to L^q(0,1;Y)$$
 $$
\xymatrix{
L^p_T(\R,X) \ar[r]^{H_1}  & L^q_T(\R,Y) \ar[d]^{(J_q)^{-1}} \\
L^p(0,T;X) \ar[r]^{g} \ar[u]^{J_p} & L^q(0,T;Y)  \\
}
$$
is well-defined. Moreover, we have $g(\omega)(s)=f(s,\omega(s))$ for $s\in(0,T)$.
We obtain {\bf (H4)} from Theorem \ref{th6} applied to the restriction of the function $f$ on $(0,T)\times X$.
Then \eqref{eq62} holds and the claim is proved.
\end{proof} 

An immediate consequence of Corollary \ref{cor5} is the following result.

\begin{corollary}
\label{cor12}
We assume that  the Banach spaces $X$ and $Y$ are separable.
For a  continuous map $f:X\to Y$, we consider the Nemytskii operator $\mathcal{N}_{f}$ defined by
$\mathcal{N}_{f}(u)=f\circ u$.
Then  the following assertions hold.
\vskip 2 mm
{\bf i)} The  Nemytskii operator $\mathcal{N}_{f}$ maps $S_{aa}^p(\R,X)$ into $S_{aa}^q(\R,Y)$  if and only if \eqref{eq50} holds. In this case $\mathcal{N}_{f}$ is continuous.
\vskip 2 mm
{\bf ii)} The  Nemytskii operator $\mathcal{N}_{f}$ maps $S_{ap}^p(\R,X)$ into $S_{ap}^q(\R,Y)$  if and only \eqref{eq50} holds.
In this case $\mathcal{N}_{f}$  is continuous.
\end{corollary}

\subsection{Proof of Theorem \ref{th5} and Theorem \ref{th11}}
\label{vi-ii}

\subsubsection{Proof of Theorem \ref{th5}}

For that we use the following lemma.

\begin{lemma}
\label{lem2}
Suppose that {\bf (H2)-(H3)} hold. Assume that $f(\cdot,x)\in BS^q(\R,Y)$ for all $x\in X$.
Let $(\omega_k)_{k\geq0}$ be a sequence in $L^p(0,1;X)$. If $\omega_k\to\omega$ in $L^p(0,1;X)$ as $k\to+\infty$, then for each $\eps>0$,
$$\lim_{k\to\infty}\sup_{n\in\Z}
\Big(\rm{meas}\left( \set{\theta\in (0,1)}{\norm{f(n+\theta,\omega_k(\theta)) - f(n+\theta,\omega(\theta))}>\eps}\right) \Big)=0 .$$
\end{lemma}

\begin{proof}
We fix $\eps>0$. Let us denote by
$$A^n_{k}=\set{\theta\in(0,1)}{\norm{f(n+\theta,\omega_k(\theta)) - f(n+\theta,\omega(\theta))}\leq\eps} , \quad k\in\N , \quad n\in\Z .$$
It is a Lebesgue measurable subset of $(0,1)$, since $f(n+\cdot,x)$ is strongly measurable on $(0,1)$ and $f(n+\theta,\cdot)$ is continuous on $X$.
We want to show that for a given $r>0$ there exists $k_0\in\N$ such that 
\begin{equation}
\label{eq24}
k\geq k_0\implies\sup_{n\in\Z}\Big(\rm{meas}\left((0,1)\right)\setminus A^n_k\Big)<r .
\end{equation}
The set $\mathcal{H}=\set{\omega_k}{k\in\N}\cup\{\omega\}$ is a compact subset of $L^p(0,1;X)$, since $\omega_k\to\omega$ in $L^p(0,1;X)$ as $k\to+\infty$. It follows that $\mathcal{H}$ is tight (see e.g., \cite[Corollary 3.3]{Di-Ma}), then there exists a compact set $K\subset X$ such that
$$\rm{meas}(\set{\theta\in(0,1)}{\omega(\theta)\notin K})<\frac{r}{4}$$
and
$$\forall k\in\N,\quad\rm{meas}(\set{\theta\in(0,1)}{\omega_k(\theta)\notin K})<\frac{r}{4} .$$
If we denote by
$$B_{k}=\set{\theta\in(0,1)}{\omega_k(\theta)\in K \text{ and } \omega(\theta)\in K} , \quad k\in\N ,$$
we have 
\begin{equation}
\label{eq25}
\rm{meas}((0,1)\setminus B_k)<\frac{r}{2}\quad\text{for all }k\in\N .
\end{equation}
If we denote by
$$C_{\delta}=\set{t\in\R}{x_1, x_2\in K , \,\norm{x_1-x_2}\leq \delta \implies\norm{f(t,x_1)-f(t,x_2)}\leq \eps}$$
and $$C^n_{\delta}=\set{\theta\in(0,1)}{x_1, x_2\in K , \,\norm{x_1-x_2}\leq \delta \implies\norm{f(n+\theta,x_1)-f(n+\theta,x_2)}\leq \eps}$$
for $\delta>0$ and $n\in\Z$, we have $\ds (n,n+1)\cap C_{\delta} = n+C^n_{\delta}$.
By using Hypothesis {\bf (H3)}, there exist $\delta_*>0$ and a measurable set $N\subset \R$ such that 
$\ds\sup_{n\in\Z}\Big(\rm{meas}\left((n,n+1)\cap N\right)\Big)<\frac{r}{4}$ and
$\R\setminus N \subset C_{\delta_*}$.
We deduce that $\ds (n,n+1)\setminus C_{\delta_*} \subset (n,n+1)\cap N$ and it follows
$\ds\sup_{n\in\Z}\Big(\rm{meas}\left((n,n+1)\setminus C_{\delta_*}\right)\Big)<\frac{r}{4}$.
From $\ds (n,n+1)\setminus C_{\delta_*} = n+ \left((0,1)\setminus C_{\delta_*}^n\right)$
and by using the invariance by translation of the Lebesgue's measure on $\R$, we obtain
\begin{equation}
\label{eq26}
\sup_{n\in\Z}\Big(\rm{meas}\left((0,1)\setminus C_{\delta_*}^n\right)\Big)<\frac{r}{4} .
\end{equation}
Let us denote by
$$D_{k}=\set{\theta\in(0,1)}{\norm{\omega_k(\theta) - \omega(\theta)}\leq\delta_*} , \quad k\in\N .$$
Then there exists $k_0\in\N$ such that 
\begin{equation}
\label{eq27}
k\geq k_0\implies        \rm{meas}\left((0,1)\setminus D_k\right)<\frac{r}{4} ,
\end{equation}
since $\omega_k\to\omega$ in $L^p(0,1;X)$ as $k\to+\infty$.
\vskip 2 mm
   We have $B_k\cap C^n_{\delta_*}\cap D_k\subset A^n_k$, then
from \eqref{eq25}-\eqref{eq27} we deduce that \eqref{eq24} holds and the claim is proved.
\end{proof}

\vskip    2mm \noindent\textit{{\bf Proof of Theorem \ref{th5}.}}
{\bf i)}
For that, we use Theorem \ref{th2}, by proving that the function $H:L_1^p(\R,X)\to BS^q\R,Y)$ defined by
$H(u)(t) = f(t,u(t))$ for $t\in\R$
is well-defined and continuous. 
For $u\in L_1^p(0,1;X)$, from Hypothesis {\bf (H1)} we have  
$$\sup_{t\in\R}\left(\int_{0}^{1}\norm{f(t+\theta,u(t+\theta)}^q\,d\theta\right)^{\frac{1}{q}} \leq a\norm{u}^{\frac{p}{q}}_{S^p} + \norm{b}_{S^q}<+\infty ,$$
then the function $H$ is well-defined. Now we state that $H$ is continuous.
For that we consider a sequence $(u_k)_{k\geq0}$ in $L_1^p(\R,X)$ such that $u_k\to u$ in $L_1^p(\R,X)$ as $k\to+\infty$. 
We want to show $H(u_k)\to H(u)$ in $BS^q(\R,Y)$ as $k\to+\infty$, for that we use Proposition \ref{prop2}.
From Lemma \ref{lem2}, we have
$$\forall\eps>0,\quad\lim_{k\to\infty}\sup_{n\in\Z}
\Big(\rm{meas}\left( \set{\theta\in (0,1)}{\norm{f(n+\theta,u_k(\theta)) - f(n+\theta,u(\theta))}>\eps}\right) \Big)=0 .$$
By using the invariance by translation of the Lebesgue's measure on $\R$ and the fact that $u_k$ and $u$ are $1$-periodic, we deduce that
$$\forall\eps>0,\quad\lim_{k\to\infty}\sup_{n\in\Z}
\Big(\rm{meas}\left( \set{s\in (n,n+1)}{\norm{H(u_k)(s) - H(u)(s)}>\eps}\right) \Big)=0 .$$
Moreover for a measurable set $E\subset\R$, we have  
$$\sup_{t\in\R}\Big(\rm{meas}\left( E\cap(t,t+1)\right)\Big) \leq 2 \sup_{n\in\Z}\Big(\rm{meas}\left( E\cap(n,n+1)\right)\Big) ,$$
then
\begin{equation}
\label{eq17}
\forall\eps>0,\quad\lim_{k\to\infty}\sup_{t\in\R}
\Big(\rm{meas}\left( \set{s\in (t,t+1)}{\norm{H(u_k)(s) - H(u)(s)}>\eps}\right) \Big)=0 .
\end{equation}
From Hypothesis {\bf (H1)} it follows that for  a measurable set $E\subset (t,t+1)$  
\begin{equation}
\label{eq18}
\left(\int_E\norm{H(u_k)(s)}^q \,ds \right)^{\frac{1}{q}} \leq a\left(\int_E\norm{u_k(s)}^p \,ds \right)^{\frac{1}{q}} + \left(\int_E\norm{b(s)}^q \,ds \right)^{\frac{1}{q}}  .
\end{equation}
The set $\set{u_k}{k\in\N}$ is relatively compact in $S_{aa}^p(\R,X)$, since $u_k\to u$ in $L_1^p(\R,X)\subset S_{aa}^p(\R,X)$ as $k\to+\infty$.  Moreover the finite set $\{b\}$ is compact in $S_{aa}^p(\R,X)$. Then by using Proposition \ref{prop3}, the sets $\set{u_k}{k\in\N}$ and $\{b\}$ are Stepanov $q$-uniformly integrable.
From \eqref{eq18}, we deduce that the subset  $\set{H(u_k)}{k\in\N}$ of $BS^q(\R,Y)$ is Stepanov $q$-uniformly integrable.
By using   \eqref{eq17} and Proposition \ref{prop2},
we deduce that $H(u_k) \to H(u)$ in $BS^q(\R,Y)$ as $k\to+\infty$.
\vskip 2 mm
{\bf ii)} The proof is similar to i) by using Theorem \ref{th3}.
\hfill$\Box$

\subsubsection{Proof of Theorem \ref{th11}}

For that we use the following lemma.

\begin{lemma}
\label{lem7}
Let $u:\R\to X$ be a map.
\vskip 2 mm
{\bf i)} If $v\in AP(\R,Y)$ and  $u$, $v$ satisfy
\begin{equation}
\label{eq19}
\forall\eps>0,\, \exists M>0,\,\forall t_1, t_2\in\R, \quad \norm{u(t_1) - u(t_2)}_X \leq \frac{\eps}{2} + M \norm{v(t_1) - v(t_2)}_Y ,
\end{equation}
then $u\in AP(\R,X)$. 
\vskip 2 mm
{\bf ii)} If $v\in AA(\R,Y)$ and $u$, $v$ satisfy \eqref{eq19}, then $u\in AA(\R,X)$. 
\end{lemma}

\begin{proof} 
From \eqref{eq19}, we deduce that if $v$ is continuous, then $u$ is also one.
\vskip2mm
{\bf i)} For $u\in C(\R,X)$, let us denote by $\ds\mathcal{P}(u,\eps)$ the set of $\eps$-almost periods of $u$. From \eqref{eq19}, we deduce that $\ds \mathcal{P}(v,\frac{\eps}{2M}) \subset \mathcal{P}(u,\eps)$, then i) holds.
\vskip2mm
{\bf ii)} By assumption, $v\in AA(\R,Y)$, then for all sequence of real numbers $(t_{n}^{\prime})_{n\in\N}$ admits  a subsequence  denoted by $(t_{n})_{n\in\N}$ such that
\begin{equation}
\label{eq59}
\forall t\in\R,\quad \underset{n\rightarrow\infty}{\lim}v(t+t_{n})  = v_*(t)\quad\text{and}\quad
\underset{k\rightarrow\infty}{\lim}v_*(t-t_{n})=v(t) .
\end{equation}
From  \eqref{eq19}  for all $\eps>0$, there exists $M>0$ such that
\begin{equation}
\label{eq71}
\forall t\in\R,\quad\norm{u(t+t_n) - u(t+t_m)}_X \leq \frac{\eps}{2} + M \norm{v(t+t_n) - v(t+t_m)}_Y .
\end{equation}
We deduce that $(u(t+t_n))_{n\in\N}$ is a sequence of Cauchy, then there exists $u_*$ such that 
\begin{equation}
\label{eq58}
\forall t\in\R,\qquad \underset{n\rightarrow\infty}{\lim}u(t+t_{n})  = u_*(t) .
\end{equation}
We fix $\eps>0$. By passing to the limit as $m\to\infty$ on \eqref{eq71}, we obtain
\begin{equation*}
\forall t\in\R,\quad\norm{u(t+t_n) - u_*(t)}_X \leq \frac{\eps}{2} + M \norm{v(t+t_n) - v_*(t)}_Y ,
\end{equation*}
then by replacing $t$ by $t-t_n$, we have 
\begin{equation}
\label{eq5}
\forall t\in\R,\quad\norm{u(t) - u_*(t-t_n)}_X \leq \frac{\eps}{2} + M \norm{v(t) - v_*(t-t_n)}_Y .
\end{equation}
From \eqref{eq59} and \eqref{eq5}, we deduce that $\ds\limsup_{n\to+\infty}\norm{u(t) - u_*(t-t_n)}_X\leq\frac{\eps}{2}$ for each $\eps>0$, then 
\begin{equation}
\label{eq65}
\forall t\in\R,\qquad\underset{k\rightarrow\infty}{\lim}u_*(t-t_{n})=u(t) .
\end{equation}
From \eqref{eq58} and \eqref{eq65}, we have $u\in AA(\R,X)$. 
\end{proof} 

\vskip    2mm \noindent\textit{{\bf Proof of Theorem \ref{th11}.}}
For this, we prove that the hypotheses of Theorem \ref{th5} are equivalent to those of Theorem \ref{th11}. Suppose that {\bf (H1)} hold. We have to prove that the following assertions are equivalent:
\vskip 2 mm
{\bf i)} For all compact set $K\subset X$, $\tilde{f}_K\in S_{aa}^q(\R,C(K,Y))$ (resp. $\tilde{f}_K\in S_{ap}^q(\R,C(K,Y))$).
\vskip 2 mm
{\bf ii)} {\bf (H2)-(H3)} and $f(\cdot,x)\in S_{aa}^q(\R,Y)$ (resp. $f(\cdot,x)\in S_{ap}^q(\R,Y)$) for all $x\in X$.
\vskip 2 mm
{\bf i) $\Longrightarrow$ ii)} For each compact set $K\subset X$, $\tilde{f}_K(t) = f(t,\cdot) \in C(K,Y)$: $f(t,\cdot)$ is continuous on $K$, then $f(t,\cdot)$ is continuous on  $X$ as continuous function on each compact of $X$. So Hypothesis {\bf (H2)} holds.
Let $x\in X$. $f(\cdot,x)\in S_{aa}^q(\R,Y)$ (resp. $f(\cdot,x)\in S_{ap}^q(\R,Y)$) results of $\tilde{f}_K\in S_{aa}^q(\R,C(K,Y))$ (resp. $\tilde{f}_K\in S_{ap}^q(\R,C(K,Y))$) by setting $K=\{x\}$.
It remains to state that {\bf (H3)}  holds. 
From Proposition \ref{prop3}, 
$\tilde{f}_K$  is Stepanov tight, 
then there exist a compact  set $\mathcal{K}$ of $C(K,Y)$ and a Lebesgue
measurable set $N\subset\R$, such that $\ds\sup_{t\in\R}\Big(\rm{meas}\left((t,t+1)\cap N\right)\Big)<r$ and $\tilde{f}_K(t)\in \mathcal{K}$ for all $t\in\R\setminus N$.
Then by Ascoli Theorem, the family $\set{\tilde{f}_K(t)}{t\in\R\setminus N}$ is    equi-uniformly continuous, that is $\forall \eps>0$, $\exists \delta>0$, $\forall x_1$,  $x_2\in K$, $\forall t\in\R\setminus N$, $\norm{x_1-x_2}\leq \delta \implies \norm{f(t,x_1)-f(t,x_2)}\leq\eps$.
Therefore {\bf (H3)} holds.
\vskip2mm
{\bf ii) $\Longrightarrow$ i)} 
Let us denote by $\alpha^{K}_{\delta}$ defined by \eqref{eq57} in Remark \ref{rq2}  where $K\subset X$ is a compact set.
From Hypothesis {\bf (H1)} we deduce that
$\ds 0 \leq\alpha^K_{\delta}(t) \leq 2\left(a \sup_{x\in K}\norm{x}^{\frac{p}{q}} + b(t)\right)$ where $\ds a \sup_{x\in K}\norm{x}^{\frac{p}{q}} + b(\cdot) \in S^q_{aa}(\R)$. 
Then $\alpha^K_{\delta}\in BS^q(\R)$. 
From Proposition \ref{prop3}, the function $\ds a \sup_{x\in K}\norm{x}^{\frac{p}{q}} + b(\cdot)$ is Stepanov $p$-uniformly integrable, then $\alpha^K_{\delta}$ is also one.
From  the characterization of Hypothesis {\bf (H3)} given in i) in Remark \ref{rq2},
we deduce that
\begin{equation*}
\forall\eps>0, \qquad \lim_{\delta\to0}\sup_{t\in\R}\Big(
\rm{meas}\left( \set{s\in (t,t+1)}{\abs{\alpha^K_{\delta}(s)}\geq\eps}\right) \Big)=0 .
\end{equation*}
Then by using Proposition \ref{prop2}, we have
\begin{equation}
\label{eq64}
\lim_{\delta\to0}\norm{\alpha^K_{\delta}}_{S^q}=0 .
\end{equation}
\vskip2mm
For $x\in X$, let us denote  by $f_x$   the function $f(\cdot,x)$. Since $f_x\in L^q_{loc}(\R,Y)$,  the map
 \begin{equation*}
F:\R\times X \to  L^q(0,1;Y) \text{ defined by } F(t,x) = (f_x)^b(t) 
\end{equation*}
 is well-defined. 
 If $f_x\in S_{aa}^q(\R,Y)$  for all $x\in X$, the function
$\underline{F} : X\to AA(\R,L^q(0,1;Y))$  with $\underline{F}(x) = F(\cdot ,x) = (f_x)^b$
 is well-defined. 
 For all compact set $K\subset X$, for all $\delta>0$, for all $x_1$ and $x_2\in K$ such that $\norm{x_1-x_2}\leq\delta$, one has 
 $$\norm{\underline{F}(x_1) - \underline{F}(x_2)}=\sup_{t\in\R}\norm{F(t,x_1) - F(t,x_2)}_{L^q} \leq \norm{\alpha^K_{\delta}}_{S^q} .$$
Then by using \eqref{eq64}, we deduce that  $\underline{F}$ is continuous each compact $K\subset X$, therefore 
$\underline{F}\in C(X,AA(\R,L^q(0,1;Y)))$.
 From Theorem \ref{th8}, we can assert that for every compact set $K\subset X$,
 the function
 $$\overline{F} : \R\to C(K,L^q(0,1;Y)) \text{ defined by }\overline{F}(t) = F(t ,\cdot)$$
satisfies 
\begin{equation}
\label{eq69}
\overline{F}\in AA(\R,C(K,L^q(0,1;Y)) \quad\text{ if }\quad \forall x\in X, \, f_x\in S_{aa}^q(\R,Y) .
\end{equation}
By using Theorem \ref{th9}, we state in the same way that
\begin{equation}
\label{eq70}
\overline{F}\in AP(\R,C(K,L^q(0,1;Y)) \quad\text{ if }\quad \forall x\in X, \, f_x\in S_{ap}^q(\R,Y) .
\end{equation}
\vskip 2 mm
Fix a compact set $K\subset X$.
By using Lemma  \ref{lem6}, we have $\tilde{f}_K\in L_{loc}^q(\R,C(K,Y))$.
Fix $\eps>0$. From \eqref{eq64}, one has 
\begin{equation}
\label{eq67}
\exists\delta_*>0 \text{ such that}\quad 
\norm{\alpha^K_{\delta_*}}_{S^q}\leq\frac{\eps}{4} .
\end{equation}
Since $K$ is a compact set, there exist $x_1,x_2,\cdots,x_N$ such that $K\subset \bigcup_{i=1}^{N}B(x_i,\delta_*)$.
Let  $x\in K$. There exists $j\in\{1,2,\cdots,N\}$ such that $\norm{x-x_j}\leq\delta_*$.
From the inequality
\[
\norm{f(t_1,x)-f(t_2,x)} \leq \norm{f(t_1,x)-f(t_1,x_j)} + \norm{f_{x_j}(t_1)-f_{x_j}(t_2)} + \norm{f(t_2,x_j)-f(t_2,x)} ,
\]
we deduce that
\begin{equation*}
\sup_{x\in K}\norm{f(t_1,x)-f(t_2,x)} \leq \alpha^K_{\delta_*}(t_1) + \alpha^K_{\delta_*}(t_2)
 + \sum_{i=1}^{N}\norm{f_{x_i}(t_1)-f_{x_i}(t_2)} .
\end{equation*}
Then 
\begin{eqnarray*}
\norm{(\tilde{f_K})^b(t_1)-(\tilde{f_K})^b(t_2)}_{L^q(0,1;C(K,Y))} \\
\leq\norm{(\alpha^K_{\delta_*})^b(t_1)}_{L^q} + \norm{(\alpha^K_{\delta_*})^b(t_2)}_{L^q}
 + \sum_{i=1}^{N}\norm{F(t_1,x_i)-F(t_2,x_i)}_{L^q} \\
\leq 2 \norm{\alpha^K_{\delta_*}}_{S^q} 
 +  N\sup_{x\in K}\norm{F(t_1,x)-F(t_2,x)}_{L^q} .
\end{eqnarray*}
By using \eqref{eq67} and $\ds\norm{\overline{F}(t_1)-\overline{F}(t_2)}_{C(K,L^q(0,1;Y))}=\sup_{x\in K}\norm{F(t_1,x)-F(t_2,x)}_{L^q}$, we have
\begin{equation*}
\norm{(\tilde{f}_K)^b(t_1)-(\tilde{f}_K)^b(t_2)}_{L^q(0,1;C(K,Y))} \leq \frac{\eps}{2}
 + N\norm{\overline{F}(t_1)-\overline{F}(t_2)}_{C(K,L^q(0,1;Y))} .
\end{equation*}
From  Lemma \ref{lem7} and \eqref{eq69}, we deduce that $(\tilde{f}_K)^b\in AA(\R,L^q(0,1;C(K,Y))$ if $f_x\in S_{aa}^q(\R,Y)$ for all $x\in X$.
Then $\tilde{f}_K\in S_{aa}^q(\R,C(K,Y))$. For similar reasons with \eqref{eq70} (instead of \eqref{eq69}) we have $\tilde{f}_K\in S_{ap}^q(\R,C(K,Y))$ if $f_x\in S_{ap}^q(\R,Y)$ for all $x\in X$.
\hfill$\Box$

\section{Example}
\label{vii}

In this section, we explain why the two equivalent theorems \ref{th5} and \ref{th11} of the previous section provide an improvement and a generalization of the known results. For that we consider the following simple example: $X$ is a Banach space and $f:\R\times X\to X$ is the function defined by
$f(t,x)=\sin(a(t){\rm e}^{\norm{x}})x$ where $a\in S^p_{aa}(\R)$ (resp. $a\in S^p_{ap}(\R)$). 
We establish that Theorem \ref{th11} allows us to conclude,  then we show that the assumptions of the known results \cite{An-Pe, Be-Ch-Me-Ra-Sm, Di-Hu-Li-Xi, Di-Li-Xi, Lo-Di, NG-Di} are not all verified.
 \vskip 2 mm
First we show that the assumptions of Theorem \ref{th11} are  satisfied
for this example.
It is obvious that Hypothesis {\bf (H2)} holds and that it is the same for Hypothesis {\bf (H1)} with $p = q\geq 1$. 
We have also $f(\cdot,x)\in L^{p}_{loc}(\R,Y)$. Let $K\subset X$ be a compact set. Then from Lemma \ref{lem6}, we have $\tilde{f}_K\in L^{p}_{loc}(\R,C(K,Y))$.
From  the following inequality
$$\norm{f(t_1,x)-f(t_2,x)} \leq \norm{x}{\rm e}^{\norm{x}} \abs{a(t_1)-a(t_2)}$$
with $R=\sup_{x\in K}\norm{x}<+\infty$, we obtain
$$\sup_{x\in K}\norm{f(t_1,x) - f(t_2,x)}  \leq R{\rm e}^{R} \abs{a(t_1)-a(t_2)} .$$ 
Then we have 
$$\norm{(\tilde{f}_K)^b(t_1) - (\tilde{f}_K)^b(t_2)}_{L^p}   \leq R{\rm e}^{R} \norm{a^b(t_1)-a^b(t_2)}_{L^p}  .$$ 
By using Lemma \ref{lem7}, we deduce  that $\tilde{f}_K\in S_{aa}^p(\R,C(K,Y))$ (resp. $\tilde{f}_K\in S_{ap}^p(\R,C(K,Y))$) if $a\in S^p_{aa}(\R)$ (resp. $a\in S^p_{ap}(\R)$). 
Then all assumptions of Theorem \ref{th11} are fulfilled.
\vskip 2 mm
Secondly in the literature \cite{Be-Ch-Me-Ra-Sm, Di-Hu-Li-Xi, Di-Li-Xi, Lo-Di, NG-Di}, except Andres et al.   \cite{An-Pe}, the authors use the Lipschitzian condition \eqref{eq49}  to state that the operator of Nemytskii maps $S_{aa}^p(\R,X)$ into itself  
 or maps $S_{ap}^p(\R,X)$ into itself.
In our example the function $f$ does not satisfy \eqref{eq49} if $a$ is not the null function.
To see that, we choice  $t_0\in\R$ such that $a(t_0)\not=0$ and $x_0\in X$ such that 
$\norm{x_0}=1$ and we denote by $\ds \eps_k = \ln\left(1+\frac{1}{2k}\right)$, $\ds s_k = \ln\left(\frac{k\pi}{\abs{a(t_0)}}\right)$, $\ds x_k = s_kx_0$ and $\ds y_k = (s_k+\eps_k)x_0$ for $k\in\N\setminus\{0\}$. Then for $k\in\N$ enough large   such that $s_k>0$, we have
$$\frac{\norm{f(t_0,y_k)-f(t_0,x_k)}}{\norm{y_k-x_k}} = \frac{\norm{f(t_0,y_k)}}{\eps_k} = 1+\frac{s_k}{\eps_k} \to+\infty \quad (\text{ as } k \to+\infty) .$$
Then $f$ does not satisfy the Lipschitzian condition \eqref{eq49}.
\vskip 2 mm
Here we explain why results \cite[Lemma 3.2, Proposition 3.4]{An-Pe} of Andres et al.   are unusable  on our example. 
If $a(t)\not=0$, we cannot use \cite[Proposition 3.4]{An-Pe} which is described in Remark \ref{rq4}, since $\ds \sup_{x\in X}\norm{f(t,x)} = \sup_{x\in X}\abs{\sin(a(t){\rm e}^{\norm{x}})} \norm{x} = +\infty$.
By using Corollary \ref{cor12},  we can assert that the Nemytskii operator associated to the function 
$g:X\to \R$ defined by $g(x)={\rm e}^{\norm{x}}$ does not map $S_{aa}^p(\R,X)$ into $S_{aa}^p(\R)$   when the Banach space $X$ is separable. Then in the particular case where 
$a(t)=1$ for all $t\in\R$, 
we cannot use  \cite[Lemma 3.2]{An-Pe} which is described in Remark \ref{rq5} to conclude.
\vskip 2 mm
For theses reasons  Theorem \ref{th11}  provides an improvement and a generalization of these results. It is the same for the equivalent Theorem \ref{th5}.


\end{document}